\documentclass[a4paper,12pt,leqno]{amsart}

\usepackage{latexsym}
\usepackage[all]{xy}
\usepackage{amssymb} 
\usepackage{amsmath} 
\usepackage{amscd}
\usepackage{color}
\usepackage{comment}
\usepackage{mathtools}
\usepackage{fancybox}
\usepackage{pb-diagram}
\usepackage{bm}
\usepackage{stmaryrd}
\usepackage{mathrsfs}
\usepackage{mathdots}
\usepackage[colorlinks=true, citecolor=blue]{hyperref}

\definecolor{gray}{gray}{0.5}

\numberwithin{equation}{section} 

\newtheorem{theorem}{Theorem}[section]
\newtheorem{lemma}[theorem]{Lemma} 
\newtheorem{corollary}[theorem]{Corollary}
\newtheorem{proposition}[theorem]{Proposition} 
 
\newtheorem{remark}[theorem]{Remark}

\newtheorem{definition}[theorem]{Definition}

\allowdisplaybreaks[3]

\def\C{\mathbb C}
\def\R{\mathbb R}
\def\Q{\mathbb Q}
\def\Z{\mathbb Z}

\DeclareMathOperator{\Hom}{Hom}

\DeclareMathOperator{\Ker}{Ker}

\DeclareMathOperator{\codim}{codim}

\DeclareMathOperator{\Ad}{Ad}

\newcommand{\RL}{\Lambda_{r}}
\newcommand{\WL}{\Lambda}
\newcommand{\CRL}{{\Lambda_{r}^{\vee}}}
\newcommand{\CWL}{\Lambda^{\vee}}
\newcommand{\rkg}{n}
\newcommand{\fan}{\Sigma}
\newcommand{\qp}[1]{q_{\alpha_{#1}}}
\newcommand{\nil}{e}
\newcommand{\RY}[2]{Y^{\circ}_{#1,#2}}
\newcommand{\Ry}[2]{C_{#1,#2}}

\newcommand{\Gad}{G^{\text{\rm ad}}}
\newcommand{\Bad}{B^{\text{\rm ad}}}

\newcommand{\Tad}{T^{\text{\rm ad}}}
\newcommand{\Uad}{U^{\text{\rm ad}}}
\newcommand{\Umad}{(U^-)^{\text{\rm ad}}}
\newcommand{\xad}{x'}
\newcommand{\yad}{y'}
\newcommand{\X}[1]{X(\fan)_{{#1}}^{\circ}}
\newcommand{\Xp}[1]{X(\fan)_{{#1};> 0}^{\circ}}

\newcommand{\shapo}[2]{\langle{#1}\hspace{1pt},\hspace{1pt}{#2}\rangle}

\begin{document}
\text{}\vspace{-10pt}
\title[Peterson variety and strongly dominant weight polytope]{Totally nonnegative Peterson variety \\and strongly dominant weight polytope}

\keywords{Peterson variety, total positivity, strongly dominant weight polytope}

\subjclass[2020]{14M15, 14M25, 05E14, 20G20, 15B48}

\author {Hiraku Abe}
\address{Faculty of Science, Department of Applied Mathematics, Okayama University of Science, 1-1 Ridai-cho, Kita-ku, Okayama, 700-0005, Japan}
\email{hirakuabe@globe.ocn.ne.jp}

\author {Tao Gui}
\address{Beijing International Center for Mathematical Research, Peking University, No. 5 Yiheyuan Road, Haidian District, Beijing 100871, P.R. China}
\email{guitao18@mails.ucas.ac.cn} 

\author {Haozhi Zeng}
\address{School of Mathematics and Statistics, Huazhong University of Science and Technology, Wuhan, 430074, P.R. China}
\email{zenghaozhi@icloud.com} 

\begin{abstract}
We study the totally nonnegative part of the Peterson variety in arbitrary Lie type and establish its connection to the strongly dominant weight polytope. In particular, we prove that the totally nonnegative part of the Peterson variety is a regular CW-complex, which is homeomorphic to a cube as a cell-decomposed space. This confirms a conjecture of Rietsch for all Lie types.
\end{abstract}

\maketitle

\section{Introduction}\label{sec Introduction}
Let $G$ be a semisimple algebraic group over $\C$ with a Borel subgroup $B\subset G$. The Peterson variety $Y$ is a certain remarkable subvarieties of the flag variety $G/B$, introduced by Dale Peterson \cite{peterson1997quantum} to realize quantum cohomology rings of all Langlands dual partial flag varieties geometrically. By using the geometry of $Y$, he discovered a connection of the quantum cohomology of those flag varieties with the homology of the affine Grassmannian $\mathcal{G} r_{G^{\vee}}$ of the Langlands dual group $G^{\vee}$, which was verified by Lam--Shimozono in \cite{lam2010quantum}. It is also closely related to the wonderful compactification of a certain unipotent subgroup of $G$ \cite{Bualibanu2017Peterson}. The geometry and topology of the Peterson variety have been studied extensively, see, for example, \cite{abe2024geometry,abe2023peterson,Bualibanu2017Peterson,Goldin2024positivity,gui2025structure,harada2015equivariant,kostant1996flag}.

The theory of total positivity for reductive algebraic groups was pioneered by Lusztig \cite{lusztig1994total} as a broad extension of the classical theory of Schoenberg and Gantmacher--Krein on total positivity for matrices. It has close connections to cluster algebras \cite{fomin2010total}, KP equations \cite{kodama2014kp,kodama2011kp}, Poisson geometry \cite{lu2020bott}, and the physics of scattering amplitudes \cite{williams2021positive}. Lusztig \cite{lusztig1994total,lusztig1998total} also defined the totally positive and the totally nonnegative parts of the flag variety, which naturally induces the corresponding definitions for subvarieties of the flag variety. Of particular interests are so-called {\it regularity theorems} on CW-complex structures on the totally non-negative parts of these varieties, see \cite{bao2024product,galashin2022regularity,hersh2014regular} for the most recent developments. Note that convex polytopes are prototypical examples of regular CW
complexes and the topology of a regular CW complex is completely determined by the combinatorial
structure of its associated cell closure poset \cite{bjorner1984posets}.

The interaction bewteen the Peterson variety and the total positivity was first studied by Rietsch. In \cite{rietsch2003totally}, she used Peterson's theory in type A to obtain a parametrization of totally nonnegative Toeplitz matrices. In \cite{Rietsch2006}, she gave a mirror construction of the totally nonnegative part of the Peterson variety $Y_{\geq 0}$ in type A and obtained its cell decomposition. She also showed in \cite{Rietsch2006} that the totally nonnegative part of the Peterson variety in type A is contractible. Based on the structure of the cells, she conjectured that as a cell decomposed space $Y_{\geq 0}$ is homeomorphic to a cube. In \cite{abe2025totally}, the first and the third authors of this paper gave a proof of Rietsch's conjecture in type A by using toric geometry closely related to the Peterson variety and concrete computations.

In this paper, we study the totally nonnegative part of the Peterson variety in general Lie types. Now we summarize our main results. By intersecting the Bruhat and opposite Bruhat decompositions of the flag variety $G/B$ and the Peterson variety $Y$, one can obtain the Richardson stratification of $Y$. In Proposition \ref{prop description of YKJ}, we give a description of the Richardson strata of $Y$ in terms of certain functions $\Delta_{\varpi_i}$. Using the positivity of these functions on the totally nonnegative part (Proposition \ref{prop Delta is nonnegative}), we give a description of the Richardson strata of the totally nonnegative part $Y_{\geq 0}$ (Proposition \ref{prop description for R KI>0 2}). Through a particular morphism $\Psi: Y \rightarrow X(\Sigma)$ constructed by the first and the third authors of this paper in \cite{abe2023peterson}, the Peterson variety $Y$ is connected to a particular projective toric orbifold $X(\Sigma)$. We show that this morphism can be restricted to the nonnegative parts $\Psi_{\geq 0}: Y_{\geq 0} \rightarrow X(\Sigma)_{\geq 0}$, which sends the Richardson strata of $Y_{\geq 0}$ to the toric orbit strata of $X(\Sigma)_{\geq 0}$ (Proposition \ref{prop restrict}). We prove that $\Psi_{\geq 0}: Y_{\geq 0} \rightarrow X(\Sigma)_{\geq 0}$ is actually a homeomorphism (Theorem \ref{thm homeo}). Since the torus orbit decomposition of the a projective toric variety gives rise to a cell decomposition of its nonnegative part (see Proposition \ref{prop nonnengative stratum parametrization} in our case), it follows that the Richardson stratification of $Y_{\geq 0}$ is actually a cell decomposition. Note that the moment map restricts to a cell-preserving homeomorphism between $X(\Sigma)_{\geq 0}$ and its moment polytope. Here, the moment polytope of $X(\Sigma)$ is the strongly dominant weight polytope, which is defined to be the intersection of the dominant Weyl chamber and a weight polytope associated with a regular weight. Since the strongly dominant weight polytope is proved to be combinatorially equivalent to a cube\footnote{Topologically, this is equivalent to the existence of a (piecewise linear) homeomorphism between the strongly dominant weight polytope and the standard cube, which restricts to homeomorphisms between their facets (and hence all the faces).} \cite{burrull2023strongly}, we deduce the following main theorem of this paper, which conforms a conjecture of Rietsch\footnote{We thank Konstanze Rietsch for private communication of her conjecture.} for arbitrary Lie type.

\begin{theorem} \label{thm main}
    The totally nonnegative part of the Peterson variety is homeomorphic to the strongly dominant weight polytope as a cell-decomposed space. In particular, it is a regular CW-complex, which is homeomorphic to a cube.
\end{theorem}

Now we give some final remarks. Firstly we remark that the functions $\Delta_{\varpi_{i}}$ and $\qp{i}$ appearing in our map $\Psi_{\geq 0}: Y_{\geq 0} \rightarrow X(\Sigma)_{\geq 0}$ are actually affine Schubert classes and quantum parameters in Peterson's theory, see \cite[Remark 6.5]{lam2016total} and \cite[Remark~3.3.7]{rietsch2008mirror}. In \cite[Theorem 7.3]{lam2016total}, Lam--Rietsch essentially used the functions $\Delta_{\varpi_{i}}$ (up to a certain power because they work on adjoint type while we work on the simply-connected case, see Proposition \ref{prop LR and our}) to give a parametrization of the totally nonnegative part of an affine piece of the Peterson variety. We remark that this parametrization is important in the proof of the bijectivity of $\Psi_{\geq 0}: Y_{\geq 0} \rightarrow X(\Sigma)_{\geq 0}$. By analyzing a crucial connection with the affine Grassmannian and geometric Satake equivalence, Lam--Rietsch also showed that different notions of positivity---quantum Schubert and quantum parameter positivity, Lusztig's total positivity, and affine Schubert positivity---on an affine part of the Peterson variety all coincide. Finally, note that $Y_{\geq 0}$ and $X(\Sigma)_{\geq 0}$ are defined in quite different ways; the totally nonnegative part $Y_{\geq 0}$ is defined in terms of root datum of $G$ whereas $X(\Sigma)_{\geq 0}$ is defined in terms of semigroup algebras. It is quite remarkable that so many apparently different notions of positivity coincide in this setting.

This paper is organized as follows. In Section \ref{sec Notation}, we fix some notations which we use throughout this paper. In Section \ref{sec Richardson strata}, we recall the definition of the Peterson variety $Y$ and the construction of its Richardson strata. In Section \ref{sec Totally nonnegative Richardson strata}, we study the totally nonnegative part $Y_{\geq 0}$ and give a description of its Richardson strata. In Section \ref{sec toric}, we study the toric orbifold $X(\Sigma)$ whose moment polytope is the strongly dominant weight polytope. In Section \ref{sec map}, we recall the definition of the morphism $\Psi: Y \rightarrow X(\Sigma)$, and reduce the fact that it induces a homeomorphism $\Psi_{\geq 0}: Y_{\geq 0} \rightarrow X(\Sigma)_{\geq 0}$ to the key claim (Theorem \ref{thm LR in our setting}). Sections \ref{sec splittings} and \ref{sec LR in our setting} are devoted to give a proof of this key claim, which completes the proof of Theorem \ref{thm main}. In the Appendix, we provide proofs of two well-known lemmas for the reader because of the lack of suitable references.

\vspace{10pt}

\subsection{Acknowledgments}
We are grateful to Thomas Lam, Changzheng Li, and Zhi L$\ddot{\text{u}}$ for valuable discussions. We would particularly like to thank Konstanze Rietsch for telling us her conjecture. This research was partially done while the authors were visiting the  school of Mathematical Sciences in Fudan University.
The first author is supported in part by JSPS Grant-in-Aid for Scientific Research(C): 23K03102. 
The second author is supported in part by NSFC:  12471309.
The third author is supported in part by NSFC:  11901218.

\vspace{30pt}

\section{Notation}\label{sec Notation}

\subsection{Set up}\label{subsec set up}

Let $G$ be a simply connected, semisimple algebraic group over $\C$ of rank $\rkg$, split over $\mathbb{R}$. Choose a Borel subgroup $B$ defined over $\mathbb{R}$ and a split maximal torus $T\subset B$. Let $W=N(T)/T$ be the Weyl group, where $N(T)$ is the normalizer of $T$ in $G$. Denoting the center of $G$ by $Z_G$, we have $Z_G\subset T$ (\cite[Sect.\ 26 Exercise 2]{humphreys2012linear}).

Let $B^-$ be the opposite Borel subgroup of $G$ so that $T=B\cap B^-$. We denote by $U\subset B$ and $U^-\subset B^-$ the unipotent radicals of $B$ and $B^-$, respectively.

Let $\mathfrak{g}$ be the Lie algebra of $G$. Similarly, we use the German typeface to denote the Lie algebra of an algebraic group; for example, $\mathfrak{b}$ is the Lie algebra of $B$, and $\mathfrak{u}^-$ is the Lie algebra of $U^-$, and so on.
The adjoint action of $T$ on $\mathfrak{g}$ gives us the root space decomposition 
\begin{align*}
 \mathfrak{g}=\mathfrak{t}\oplus\bigoplus_{\alpha\in\Phi}\mathfrak{g}_{\alpha}, 
 \end{align*}
where $\Phi$ is the set of roots, and $\mathfrak{g}_{\alpha}$ is the root space for $\alpha\in\Phi$.
Let $\Phi^+$ be the set of positive roots associated to $B$ and $\{\alpha_1,\ldots,\alpha_{\rkg}\}\subset\Phi^+$ the set of simple roots. Denote by $I$ the Dynkin diagram of $\Phi$ which we identify with the indexing set $\{1,\ldots,\rkg\}$ for the simple roots $\alpha_1,\ldots,\alpha_{\rkg}$ unless otherwise specified. 
For each root $\alpha\in\Phi$, we denote by $U_{\alpha}$ the root subgroup of $G$ associated to $\alpha$.

Let $\RL=\Hom(T/Z_G,\C^{\times})$ be the root lattice of $T$, and let $\WL=\Hom(T,\C^{\times})$ be the weight lattice of $T$. There is the canonical inclusion $\RL\hookrightarrow \WL$ which takes the pullback by the quotient map $T\rightarrow T/Z_G$. Regarding them as $\Z$-modules, we have
\begin{align*}
 \RL = \bigoplus_{i\in I} \Z \alpha_i\quad \text{and} \quad \WL = \bigoplus_{i\in I} \Z \varpi_i,
\end{align*}
where $\varpi_1, \ldots, \varpi_{\rkg}$ are the fundamental weights of $T$. 
Also, let $\CRL=\Hom(\C^{\times},T)$ be the coroot lattice, and let $\CWL=\Hom(\C^{\times},T/Z_G)$ be the coweight lattice. We also have the canonical inclusion $\CRL\hookrightarrow \CWL$ given by the composition with the quotient map $T\rightarrow T/Z_G$. As $\Z$-modules, we have 
\begin{align*}
 \CRL = \bigoplus_{i\in I} \Z \alpha^{\vee}_i \quad \text{and} \quad \CWL =\bigoplus_{i\in I} \Z \varpi^{\vee}_i, 
\end{align*}
where $\alpha^{\vee}_1, \ldots, \alpha^{\vee}_{\rkg}$ are the simple coroots and $\varpi^{\vee}_1, \ldots, \varpi^{\vee}_{\rkg}$ are the fundamental coweights.
These bases satisfy
\begin{align*}
 \langle \alpha_i , \varpi^{\vee}_j \rangle = \delta_{ij}
 \quad \text{and} \quad
 \langle \varpi_i , \alpha^{\vee}_j  \rangle = \delta_{ij}
 \qquad (i,j\in I)
\end{align*}
under the dual parings $\RL\times \CWL \rightarrow \Z$ and $\WL\times \CRL \rightarrow \Z$. Here, we use the same symbol $\langle \ , \ \rangle$ for the both parings by abusing notation.

\vspace{30pt}

\section{Peterson variety and its Richardson strata} \label{sec Richardson strata}

For each $i\in I$, let $e_{i}\in \mathfrak{g}_{\alpha_i}$ and $f_{i}\in \mathfrak{g}_{-\alpha_i}$ be non-zero elements, and let $\nil\in\mathfrak{g}$ be the regular nilpotent element defined by 
\begin{align}\label{eq def of e}
\nil\coloneqq\sum_{i\in I} e_{i} \in \mathfrak{g}.
\end{align}
The \textbf{Peterson variety} $Y\subseteq G/B$ is defined to be 
\begin{align*}
Y \coloneqq \left\{gB \in G/B \ \left| \ \Ad_{g^{-1}}\nil \in \mathfrak{b} \oplus \bigoplus_{i\in I}\mathfrak{g}_{-\alpha_i}\right. \right\}.
\end{align*}
It is known that $\dim_{\C}Y=|I|=n$, the rank of $G$ (\cite{peterson1997quantum,precup2013affine}).
In this section, we study a stratification of $Y$ obtained by taking intersections with Richardson cells for the flag variety.  

\vspace{10pt}

\subsection{Richardson strata of $Y$}\label{subsec Richardson strata of Y}

For $w\in W$, let
\begin{align*}
 X_{w}^{\circ} \coloneqq BwB/B
 \quad \text{and} \quad
 \Omega_{w}^{\circ} \coloneqq B^-wB/B
\end{align*}
be the Schubert cell and the opposite Schubert cell associated to $w$, respectively.
For each subset $J\subseteq I$, we denote by $w_J$ the longest element of the parabolic subgroup $W_J\subseteq W$.
The following claim is well-known. For example, the proof of \cite[Lemma~3.5]{abe2024geometry} works almost verbatim for our setting as well.

\begin{lemma}
For $w\in W$, the following are equivalent;
\begin{itemize}
 \item[(1)] $X^{\circ}_w\cap Y\ne\emptyset$,
 \item[(2)] $\Omega^{\circ}_w\cap Y\ne\emptyset$,
 \item[(3)] $\dot{w}B\in Y$, i.e., $w=w_J$ for some $J\subseteq I$,
\end{itemize}
where $\dot{w}\in N_G(T)$ is a representative for $w\in W$.
\end{lemma}

\vspace{10pt}

Recall that we have $\Omega^{\circ}_v\cap X^{\circ}_w\ne\emptyset$ if and only if $v\le w$. Recall also that $w_K\le w_J$ is equivalent to $K\subseteq J$ (\cite[Chap.~3, Exercise~5]{bjorner2005combinatorics}).
So we obtain the following.

\begin{corollary}\label{cor Richardson stratification of Y}
We have
\begin{align*}
 Y = \bigsqcup_{K\subseteq J\subseteq I} \RY{K}{J}, 
 \qquad \RY{K}{J}\coloneqq Y\cap \Omega_{w_K}^{\circ}\cap X_{w_J}^{\circ}.
\end{align*}
\end{corollary}

\vspace{10pt}

We call each $\RY{K}{J}$ a \textit{Richardson stratum} of $Y$.
The goal of this section is to give a concrete description of $\RY{K}{J}$.

We first study the intersection $Y\cap X_{w_J}^{\circ}$.
For $J\subseteq I$, let $L_J$ be the standard Levi subgroup of $G$ associated to $J\subseteq I$ which satisfies $T\subseteq L_J$.
Recall that the Schubert cell associated to the longest element $w_J$ of $W_J$ is given by
\begin{align}\label{eq Schubert cell for wJ 0}
 X_{w_J}^{\circ} 
 =B\dot{w}_JB/B
 = \{x\dot{w}_J B \in G/B \mid x \in U\cap L_J \} ,
\end{align}
where $\dot{w}_J\in N_G(T)$ is a representative of $w_J\in W$ and $U$ is the unipotent radical of $B$.
We set
\begin{align*}
 G_J \coloneqq (L_J,L_J) \subseteq L_J,
\end{align*}
where $(L_J,L_J)$ is the commutator subgroup of $L_J$.
Then $G_J$ is a semisimple algebraic group (\cite[Sect.~8.1.6]{springer1998linear}) of rank $|J|$.
Since $G$ is simply connected, so is $G_J$ (e.g.\ \cite[Corollary~9.5.11]{conrad2020reductive}).
Recalling that $L_J$ is the centralizer of a torus $(\cap_{j\in J} \Ker\alpha_j)^{\circ}$, the intersection $B'\coloneqq B\cap L_J$ is a Borel subgroup of $L_J$ (\cite[Corollary~22.4]{humphreys2012linear}). Hence,
\begin{align*}
 B_J \coloneqq B\cap G_J = B'\cap G_J
\end{align*}
is a Borel subgroup of $G_J$ since $G_J$ is a connected normal closed subgroup of $L_J$. 
In fact, the normality $G_J\trianglelefteq L_J$ ensures that the connected component $(B'\cap G_J)^{\circ}$ is a Borel subgroup of $G_J$, so it follows from \cite[Corollary~23.1A]{humphreys2012linear} that $B'\cap G_J$ is in fact connected.
Since we have $T\subseteq L_J$, it also follows that
\begin{align*}
 T_J \coloneqq T\cap G_J 
\end{align*}
is a maximal torus of $G_J$ (\cite[Sect.\ 27, Exercise~7]{humphreys2012linear}) contained in $B_J$.
%Note that $T_J$ is a $|J|$-dimensional torus generated by $\alpha^{\vee}_j(\C^{\times})\subseteq T$ for $j\in J$. 
Since we have $\alpha^{\vee}_j(\C^{\times})\subseteq G_{\{j\}}\subseteq G_J$ for $j\in J$ \cite[Sect.~7.3]{springer1998linear}, it follows that $T_J$ is generated by the one dimensional subtori $\alpha^{\vee}_j(\C^{\times})\subseteq T$ for $j\in J$.

The two reductive groups $G_J$ (with $T_J$) and $L_J$ (with $T$) have the same root system (\cite[Sect.~8.1.6]{springer1998linear}), and we also have $U\cap G_J \subseteq U\cap L_J$.
We know that the latter group $U\cap L_J$ is connected since it is the unipotent part of the Borel subgroup $B'=B\cap L_J$ (of $L_J$) appeared above (\cite[Theorem~19.3]{humphreys2012linear}). Thus, it follows that $U\cap G_J=U\cap L_J$ because of the dimension and connectedness.
So we set
\begin{align*}
 U_J\coloneqq U\cap G_J=U\cap L_J.
\end{align*}
Then, we have from \eqref{eq Schubert cell for wJ 0} that
\begin{align}\label{eq Schubert cell for wJ}
 X_{w_J}^{\circ} 
 = \{x\dot{w}_J B \in G/B \mid x \in U_J \}.
\end{align}
We set
\begin{align*}
 U^e \coloneqq U\cap Z_G(e),
\end{align*}
where $Z_G(e)$ denotes the centralizer of $e\in\mathfrak{g}$ in $G$.
For $J\subseteq I$, we also set
\begin{align}\label{eq def of UJeJ}
 (U_J)^{e_J} \coloneqq U_J\cap Z_{G_J}(e_J),
\end{align}
where $e_J\coloneqq \sum_{j\in J}e_j$.
The next claim is due to B$\breve{\text{a}}$libanu. She used semisimple algebraic groups of adjoint type, but her proof works verbatim for simply connected case as well.

\begin{lemma}$($\cite[Proposition~6.3]{Bualibanu2017Peterson}$)$\label{lem Y cap X circ}
For $J\subseteq I$, we have
\begin{align*}
 Y\cap X_{w_J}^{\circ}
 = \left\{ x\dot{w}_J B\in G/B \mid x\in (U_J)^{e_J} \right\},
\end{align*}
where $\dot{w}_J\in N_G(T)$ is a representative of $w_J\in W$.
\end{lemma}

\vspace{10pt}

We next study the intersection $Y\cap \Omega_{w_K}^{\circ}$ for $K\subseteq I$.
Let $i\in I$, and let $V_{\varpi_i}$ be the irreducible representation of $G$ with a highest weight vector $v_{\varpi_i}\in V_{\varpi_i}$ of weight $\varpi_i$.
We consider a function
\begin{align}\label{eq def of Delta}
 \Delta_{\varpi_i} \colon G \rightarrow \C
 \quad ; \quad
 g \mapsto (gv_{\varpi_i})_{\varpi_i} ,
\end{align} 
where $(gv_{\varpi_i})_{\varpi_i}$ denotes 
the coefficient of $v_{\varpi_i}$ for the weight decomposition of $gv_{\varpi_i}$ in $V_{\varpi_i}$.
Note that the definition of $\Delta_{\varpi_i}$ does not depend on the choice of the highest weight vector $v_{\varpi_i}$.
Denoting the closure of $\Omega_{w}^{\circ}$ by $\Omega_{w}$ for $w\in W$, it is well-known that the Schubert divisor $\Omega_{s_i}$ is given by the zero locus of $\Delta_{\varpi_i}$ (e.g.\ \cite[Lemma~5.7]{abe2023peterson}) :
\begin{align}\label{eq Omega si by equation}
 \Omega_{s_i} = \{ gB \in G/B \mid \Delta_{\varpi_i}(g)=0 \}.
\end{align}

\vspace{5pt}

\begin{lemma}\label{lem Y cap Omega circ}
For $K\subseteq I$, we have
\begin{align*}
 Y\cap\Omega_{w_K}^{\circ}
 = \left\{gB\in Y \ \left| \ 
 \begin{cases}
 \Delta_{\varpi_i}(g)=0 \quad \text{if $i\in K$}, \\
 \Delta_{\varpi_i}(g)\ne0 \quad \text{if $i\in I-K$}
 \end{cases}
 \right.
 \right\}.
\end{align*}
\end{lemma}

\begin{proof}
Let us denote the right hand side by $C_K$. 
We then have two decompositions:
\begin{align*}
 Y = \bigsqcup_{K\subseteq I} Y\cap\Omega_{w_K}^{\circ} 
 \quad \text{and} \quad
 Y = \bigsqcup_{K\subseteq I} C_K .
\end{align*}
Therefore, to show that $Y\cap\Omega_{w_K}^{\circ}=C_K$, it suffices to prove that $Y\cap\Omega_{w_K}^{\circ}\subseteq C_K$ for all $K\subseteq I$.
To this end, we have
\begin{align*}
 Y\cap\Omega_{w_K}^{\circ}
 \subseteq Y\cap\Omega_{w_K}
 =\bigsqcup_{K\subseteq L} Y\cap\Omega_{w_L}^{\circ}
 \subseteq \bigcap_{i\in K} Y\cap\Omega_{s_i} ,
\end{align*}
where the last inclusion follows since each $L$ satisfies $s_i\le w_L$ for all $i\in K(\subseteq L)$.
This means that
\begin{align}\label{eq YcapOmegaK in Deltai=0}
 Y\cap\Omega_{w_K}^{\circ}
 \subseteq \{gB\in Y \mid \Delta_{\varpi_i}(g)=0 \quad \text{for $i\in K$}\}
\end{align}
by \eqref{eq Omega si by equation}. Moreover, for each $i\in I-K$, we have $s_i\not\le w_K$ (since $i\notin K$) so that $(Y\cap\Omega_{s_i}) \cap (Y\cap\Omega_{w_K}^{\circ}) =\emptyset$.
This implies that
\begin{align*}
 Y\cap\Omega_{w_K}^{\circ}
 \subseteq Y - \bigsqcup_{i\in I-K} Y\cap\Omega_{s_i}.
\end{align*}
Hence, we obtain that
\begin{align*}
 Y\cap\Omega_{w_K}^{\circ}
 \subseteq \{gB\in Y \mid \Delta_{\varpi_i}(g)\ne 0 \quad \text{for $i\in I-K$}\}.
\end{align*}
This and \eqref{eq YcapOmegaK in Deltai=0} imply that $Y\cap\Omega_{w_K}^{\circ} \subseteq C_K$ by the definition of $C_K$.
This completes the proof.
\end{proof}

\vspace{10pt}

For $K\subseteq J\subseteq I$, recall that we have $\RY{K}{J} = Y\cap\Omega_{w_K}^{\circ} \cap X_{w_J}^{\circ}$ (Corollary~\ref{cor Richardson stratification of Y}).
Using Lemma~\ref{lem Y cap X circ} and Lemma~\ref{lem Y cap Omega circ}, we obtain the following description of Richardson strata.

\begin{proposition}\label{prop description of YKJ}
For $K\subseteq J\subseteq I$, we have
\begin{align*}
 \RY{K}{J}
 = \left\{ x\dot{w}_J B\in G/B \ \left| \ x \in (U_J)^{e_J} \text{ and }
 \begin{cases}
 \Delta_{\varpi_i}(x\dot{w}_J)=0 \quad \text{if $i\in K$}, \\
 \Delta_{\varpi_i}(x\dot{w}_J)\ne0 \quad \text{if $i\in I-K$}
 \end{cases}
 \right.
 \right\},
\end{align*}
where $(U_J)^{e_J}=U_J\cap Z_{G_J}(e_J)$.
\end{proposition}

\vspace{10pt}

\subsection{The functions $\Delta_{\varpi_i}$ on $G_J$}

For later use, let us mention some properties of our functions $\Delta_{\varpi_i}$ on the subgroup $G_J$ for $J\subseteq I$.

Let $\varpi'_i\coloneqq \varpi_i|_{T_J}$ for $i\in J$.
Then these form the set of fundamental weights of $T_J$ (with respect to the set of simple roots $\alpha_i$ for $i\in J$) since $T_J$ is the torus generated by $\alpha_j^{\vee}(\C^{\times})$ for $j\in J$. 
Denote by $V_{\varpi'_i}$ the fundamental representation of $G_J$ having the highest weight $\varpi'_i$.
Let $v'_i$ be a highest weight vector in $V_{\varpi'_i}$, and define a function $\Delta_{\varpi'_i}$ on $G_J$ by the same manner as \eqref{eq def of Delta}.
Namely, we set
\begin{align}\label{eq def of Delta'}
 \Delta_{\varpi'_i} \colon G_J \rightarrow \C
 \quad ; \quad 
 g\mapsto (gv'_i)_{\varpi'_i} ,
\end{align}
where $(gv'_i)_{\varpi'_i}$ denotes  the coefficient of $v'_i$ for the weight decomposition of $gv'_i$ in $V_{\varpi'_i}$.
The next claim seems to be well-known, but we give a proof for the reader's convenience.

\begin{proposition}\label{prop restriction of Delta}
For $J\subseteq I$, the following hold:
\begin{itemize}
 \item[$\text{\rm (i)}$] If $i\in J$, then we have $\Delta_{\varpi_i}|_{G_J}=\Delta_{\varpi'_i}$.
 \item[$\text{\rm (ii)}$] If $i\in I-J$, then $\Delta_{\varpi_i}|_{G_J}\equiv1$.
\end{itemize}
\end{proposition}

\begin{proof}
We first prove (i).
Recall that $V_{\varpi_i}$ is an irreducible $G$-module, and let $v_{\varpi_i}\in V_{\varpi_i}$ be a highest weight vector of weight $\varpi_i$.
Let us regard $V_{\varpi_i}$ as a $G_J$-module via the inclusion $G_J\hookrightarrow G$.
It is clear that $U_J$ fixes $v_{\varpi_i}$.
Let 
\begin{align*}
 V\coloneqq \text{span}_{\C} (G_J v_{\varpi_i}) \subseteq V_{\varpi_i}
\end{align*}
which is a $G_J$-submodule of $V_{\varpi_i}$.
Since $G_J$ is semisimple, $V$ is completely reducible. Hence, \cite[Proposition~31.2]{humphreys2012linear} implies that $V$ is in fact an irreducible $G_J$-module.
Since we have $v_{\varpi_i}\in V$ (which is a highest weight vector of weight $\varpi'_i$ for $T_J$), it follows that $V$ is the irreducible $G_J$-module with highest weight $\varpi'_i$. In particular, we have $V\cong V_{\varpi'_i}$ as $G_J$-modules.
Namely, we found $V_{\varpi'_i}$ in $V_{\varpi_i}$ :
\begin{align*}
 v_{\varpi_i} \in V_{\varpi'_i} \subseteq V_{\varpi_i}.
\end{align*}
Since the definition of $\Delta_{\varpi'_i}$ does not depend on the choice of a highest weight vector, this implies that $\Delta_{\varpi_i}(g)=\Delta_{\varpi'_i}(g)$ for all $g\in G_J$.

For (ii), notice that $G_J$ is generated by $U_{\pm\alpha_j}$ for $j\in J$ since $G_J$ is semisimple (\cite[Theorem~27.5~(e)]{humphreys2012linear}).
So it suffices to prove that 
\begin{align}\label{eq product 105}
 U_{\pm\alpha_j}v_{\varpi_i}=v_{\varpi_i} \qquad (j\in J).
\end{align}
The claim for $U_{\alpha_j}$ is obvious since we have $U_{\alpha_j}\subseteq U$ and $v_{\varpi_i}$ is a highest weight vector.
So we prove the claim for $U_{-\alpha_j}$.
To begin with, note that we have
\begin{align*}
 U_{-\alpha_j} = \dot{s}_{j}^{-1} U_{\alpha_j} \dot{s}_{j} ,
\end{align*}
and the $U_{\alpha_j}$ in the right hand side acts trivially on $v_{\varpi_i}$ as we saw above.
We want to compute $U_{-\alpha_j} v_{\varpi_i} = \dot{s}_{j}^{-1} U_{\alpha_j}\dot{s}_{j}v_{\varpi_i}$. Here, $\dot{s}_{j}v_{\varpi_i}$ is an eigenvector for $T$ since $\dot{s}_{j}\in N_G(T)$, and its weight is given by
\begin{align}\label{eq product 115}
 s_{j} \varpi_i = \varpi_i - \langle \varpi_i , \alpha^{\vee}_j \rangle \alpha_{j} =  \varpi_i ,
\end{align}
where the last equality holds since $j\ne i$ (due to $J\cap(I-J)=\emptyset$).
Therefore, we have $\dot{s}_{j}v_{\varpi_i}=\lambda v_{\varpi_i}$ for some $\lambda\in\C^{\times}$.
Hence, we obtain that
\begin{align*}
 U_{-\alpha_j} v_{\varpi_i} 
 = \dot{s}_{j}^{-1} U_{\alpha_j} \dot{s}_{j} v_{\varpi_i}
 = \dot{s}_{j}^{-1} U_{\alpha_j} \lambda v_{\varpi_i}
 = \dot{s}_{j}^{-1} \lambda v_{\varpi_i}
 = v_{\varpi_i}.
\end{align*}
Thus, we proved \eqref{eq product 105} which completes the proof.
\end{proof}

\vspace{10pt}

\begin{remark}\label{rem on Delta being 1}
{\rm
In the description of $\RY{K}{J}$ given in Proposition~\ref{prop description of YKJ}, if we choose the representative $\dot{w}_J$ appearing there to lie in $G_J$, then we can impose additional conditions $\Delta_{\varpi_i}(x\dot{w}_J)=1$ for $i\in I-J$ by Proposition~\ref{prop restriction of Delta}~(ii).
This will be naturally achieved in the next section since we will choose specific representatives $\dot{w}$ for $w\in W$ there.
}
\end{remark}

\vspace{30pt}

\section{Totally nonnegative part of Peterson variety}\label{sec Totally nonnegative Richardson strata}

In this section, we study the totally nonnegative part of the Peterson variety.
For that purpose, we begin with fixing a pinning for $G$.

\subsection{Pinning for $G$}\label{subsec pinning}

Recall that $\alpha_i, \varpi_i\in\Hom(T,\C^{\times})$ for $i\in I$ are the simple roots and the fundamental weights, and that $\alpha^{\vee}_i, \varpi^{\vee}_i\in\Hom(\C^{\times},T/Z_G)$ for $i\in I$ are the simple coroots and the fundamental coweights (Section~\ref{subsec set up}).
If there is no confusion, we use the same symbols $\alpha_i$, $\varpi_i$, $\alpha^{\vee}_i$, $\varpi^{\vee}_i$ for their derivatives at the identity: 
\begin{align}\label{eq differential of roots and coroots}
\alpha_i\colon \mathfrak{t}\rightarrow \C, 
\quad
\varpi_i\colon \mathfrak{t}\rightarrow \C, 
\quad
 \alpha^{\vee}_i\colon \C\rightarrow \mathfrak{t}, 
\quad
 \varpi^{\vee}_i\colon  \C\rightarrow \mathfrak{t}
\end{align}
to simplify the notation.
With this understanding, we set
\begin{align*}
 h_i\coloneqq \alpha^{\vee}_i(1)\in\mathfrak{t}
\end{align*}
so that $\alpha_i(h_i)=2$. 
We now take $(e_i, f_i)_{i\in I}$ to be a set of Chevalley generators of $\mathfrak{g}$. Namely, we have $e_i\in\mathfrak{g}_{\alpha_i}$, $f_i\in\mathfrak{g}_{-\alpha_i}$ for $i\in I$, and 
\begin{align*}
 [e_i,f_i] = h_i 
\end{align*}
so that $\{e_i,h_i,f_i\}$ forms an $\mathfrak{sl}_{2}(\C)$-triple.
We fix them for the rest of this paper.

Now let $i\in I$. Using the Chevalley generators chosen above, we define parametrizations $x_i \colon \C^{\times} \rightarrow U_{\alpha_i}$ and $y_i \colon \C^{\times} \rightarrow U_{-\alpha_i}$ by 
\begin{align*}
 x_i(t) = \exp(te_i) \quad \text{and} \quad y_i(t)=\exp(tf_i)
\end{align*}
for all $t\in \C^{\times}$.
We also set
\begin{align}\label{eq choice of representative 1}
 \dot{s}_i \coloneqq y_i(1)x_i(-1)y_i(1) = x_i(-1)y_i(1)x_i(-1) ,
\end{align}
where the second equality follows from \cite[Lemma~8.1.4~(ii)]{springer1998linear}).
For a reduced expression $w=s_{i_1}\cdots s_{i_k}$, we set
\begin{align}\label{eq choice of representative 2}
 \dot{w}\coloneqq \dot{s}_{i_1}\cdots \dot{s}_{i_k} \in N_G(T).
\end{align}
This definition does not depend on the choice of a reduced expression $w=s_{i_1}\cdots s_{i_k}$ since $\{\dot{s}_{i}\}_{i\in I}$ satisfies the same braid relations as $W$ (\cite[Proposition~9.3.2]{springer1998linear}).
For the rest of this paper, we always take this representative for $w\in W$, unless otherwise specified.

\vspace{10pt}

\subsection{Totally nonnegative parts $G_{\ge0}$ and $(G/B)_{\ge0}$}\label{TNN G and G/B}

In this section, let us review the definitions of the totally nonnegative parts $G_{\ge0}$ and $(G/B)_{\ge0}$ from \cite{lusztig1994total} which is our main reference.

Let $U_{\ge0}$ be the submonoid of $U$ generated by $x_i(a)$ for $i\in I$ and $a\in\R_{\ge0}$.
Similarly, $U^-_{\ge0}$ is defined to be the submonoid of $U^-$ generated by $y_i(a)$ for $i\in I$ and $a\in\R_{\ge0}$.
Let $T_{>0}$ be the subgroup of $T$ generated by $\chi(t)$ for $\chi\in\Hom(\C^{\times},T)$ and $t\in\R_{>0}$.
Now, the \textit{totally nonnegative part} $G_{\ge0}$ is defined to be the submonoid of $G$ generated by $U_{\ge0}$, $T_{>0}$, and $U^-_{\ge0}$.

Let $w\in W$, and suppose that $w=s_{i_1}\cdots s_{i_m}$ is a reduced expression. 
Then the set
\begin{align*}
 U(w) \coloneqq \{ x_{i_1}(a_1)\cdots x_{i_m}(a_m) \in U \mid a_1,\ldots,a_m\in \R_{>0} \}
\end{align*}
does not depend on the choice of the reduced expression $w=s_{i_1}\cdots s_{i_m}$.
This gives us a partition of $U_{\ge0}$ (\cite[Corollary~2.8]{lusztig1994total}):
\begin{align*}
 U_{\ge0} = \bigsqcup_{w\in W} U(w)  .
\end{align*}
We now define 
\begin{align*}
 U_{>0}\coloneqq U(w_0) ,
\end{align*}
where $w_0$ is the longest element of $W$.
We also set $U^-_{>0}$ in a similar manner.
We now define $G_{>0}$ to be the submonoid (without $1$) of $G$ generated by $U_{>0}$, $T_{>0}$, and $U^-_{>0}$. We note that we have $\overline{U_{>0}}=U_{\ge0}$ and $\overline{G_{>0}}=G_{\ge0}$ by taking the closure with respect to the analytic topology (\cite[Proposition~4.2 and Remark~4.4]{lusztig1994total}). 
 
Let 
\begin{align*}
 (G/B)_{>0} \coloneqq \{ gB \in G/B \mid g\in G_{>0} \}=\{ uB \in G/B \mid u\in U^-_{>0} \},
\end{align*}
and set $(G/B)_{\ge0}\coloneqq \overline{(G/B)_{>0}}$.
We also define $U^-(w)$ $(w\in W)$, $(G/B^-)_{>0}$, and $(G/B^-)_{\ge0}$ in the same manner by replacing the roles of $U$ and $U^-$. 
The quotient map $G\rightarrow G/B$ restricts to $G_{>0}\rightarrow (G/B)_{>0}$ as well as $G_{\ge0}\rightarrow (G/B)_{\ge0}$.

The properties in the following lemma seem to be known for experts.

\begin{lemma}\label{lem two facts for TNN}
The following hold:
\begin{itemize}
 \item[$(1)$] $U_{>0}\cdot(G/B)_{\ge 0}\subseteq (G/B)_{\ge 0}\cap (B^-eB/B)$, \vspace{5pt}
\vspace{5pt}\item[$(2)$] $U^-(w)B/B = (G/B)_{\ge0}\cap (B^-eB/B)\cap (B\dot{w}B/B)$.
\end{itemize}
\end{lemma}

\begin{proof}
Claim (1) was proved in the proof of \cite[Corollary~10.5]{Rietsch2006} for Lie type A, and the argument works verbatim for our setting as well.

For claim (2), note that we have $U^-(w) = U^-_{\ge0}\cap B\dot{w}B$ for $w\in W$ (e.g. \cite[Lemma~6.1]{abe2025totally}). By using this equality, one can verify claim (1) by a straightforward argument
(e.g.\ \cite[Sect.\ 1.3]{rietsch1999algebraic}).
\end{proof}

\vspace{10pt}

\subsection{Totally nonnegative parts of Richardson strata}\label{subsec Totally nonnegative part of Richardson strata}

In this section, we give a description of totally nonnegative parts of Richardson strata of $Y$ (Proposition~\ref{prop description for R KI>0 2}). We begin with preparing some basic properties of the totally nonnegative part of $G/B$.

\begin{lemma}\label{lem yw0=u+b 2}
For $x\in U$, we have 
\begin{equation*}
x\dot{w}_0B\in (G/B)_{\ge0}
\quad \text{if and only if} \quad
x\in U_{\ge0}.
\end{equation*}
\end{lemma}

\begin{proof}
By \cite[Theorem~8.7]{lusztig1994total}, 
the isomorphism
\begin{equation*}
G/B^-\to G/B
\quad ; \quad 
gB^-\mapsto g\dot{w}_0B
\end{equation*}
restricts to a homeomorphism between $(G/B^-)_{>0}$ and $(G/B)_{>0}$.
In fact, for a given $uB^-\in (G/B^-)_{>0}$ with $u\in U_{>0}$, Lusztig's claim above implies that there exists $v\in U^-_{>0}$ such that $uB^-u^{-1}=vBv^{-1}$. This means that $v^{-1}u\dot{w}_0\in N_G(B)=B$ (\cite[Theorem~23.1]{humphreys2012linear}) so that $u\dot{w}_0B=vB$ belongs to $(G/B)_{>0}$. 
By taking closures, it also restricts to a homeomorphism between $(G/B^-)_{\ge0}$ and $(G/B)_{\ge0}$. 
We use this observation to prove the claim of this Lemma.

Suppose that $x\in U_{\ge0}$. 
Since $U_{\ge0}\subseteq G_{\ge0}$, we have $xB^-\in (G/B^-)_{\ge0}$ which means that 
\begin{equation*}
x\dot{w}_0B \in (G/B)_{\ge0}
\end{equation*}
by the above observation.

Conversely, suppose that $x\dot{w}_0B\in (G/B)_{\ge0}$. Let $t>0$ and set $\xi(t)\coloneqq\exp(te)$, where $e$ is the regular nilpotent element defined in \eqref{eq def of e}.
Then we have  $\xi(t)\in U_{>0}$ by \cite[Proposition 5.9 (a)]{lusztig1994total}.
Thus, we have from Lemma~\ref{lem two facts for TNN}~(1) that 
\begin{equation*}
 \xi(t)x\dot{w}_0B \in 
 (G/B)_{\ge0}\cap (B^-eB/B) .
\end{equation*}
Since we also have $\xi(t)x\dot{w}_0B\in B\dot{w}_0B/B$, we obtain 
\begin{equation*}
 \xi(t)x\dot{w}_0B\in 
 (G/B)_{\ge0}\cap (B^-eB/B)\cap (B\dot{w}_0B/B)=
 (G/B)_{>0} ,
\end{equation*}
where the last equality follows from Lemma~\ref{lem two facts for TNN}~(2) for the case $w=w_0$.
Thus, by the above observation, we obtain
\begin{equation*}
\xi(t)xB^- \in (G/B^-)_{>0}.
\end{equation*}
By the definition of $(G/B^-)_{>0}$, this means that there exists $x'\in U_{>0}$ such that $\xi(t)xB^-=x'B^-$.
Since $\xi(t)x,x'\in U$, this implies $\xi(t)x=x'$ so that $\xi(t)x\in U_{>0}$. Taking the limit $t\rightarrow 0$, we obtain $x\in \overline{U_{>0}}=U_{\ge0}$.
\end{proof}

\vspace{20pt}

In \cite[Proposition~3.2]{lam2016total}, Lam--Rietsch reformulated a result of Bernstein--Zelevinsky given in \cite[Theorem~6.9]{berenstein1997total}, which deduces the next claim.

\begin{proposition}\label{prop Delta is nonnegative}
For $x\in U_{\ge0}$ and $w\in W$, we have
\begin{equation*}
 \Delta_{\varpi_i}(x\dot{w}) \ge0
 \qquad (i\in I),
\end{equation*}
where $\Delta_{\varpi_i}$ is the function defined in \eqref{eq def of Delta}.
\end{proposition}

\begin{proof}
Let $v_{\varpi_i}$ be a highest weight vector in the fundamental representation $V_{\varpi_i}$, and $\shapo{\ }{\ }$ the Shapovalov form on $V_{\varpi_i}$ which is uniquely determined by the condition $\shapo{v_{\varpi_i}}{v_{\varpi_i}}=1$ (\cite[Sect.\ 3.14]{humphreys2008representations}).
Then our function $\Delta_{\varpi_i}$ (defined in \eqref{eq def of Delta}) can be expressed as 
\begin{align*}
 \Delta_{\varpi_i}(x\dot{w}) = \shapo{x\dot{w}v_{\varpi_i}}{v_{\varpi_i}}
 = \shapo{\dot{w}v_{\varpi_i}}{yv_{\varpi_i}}
\end{align*}
for some $y\in U^-_{\ge0}$ since the Shapovalov form is contravariant. Now the claim follows from \cite[Proposition~3.2]{lam2016total} and $U^-_{\ge0}=\overline{U^-_{>0}}$.
\end{proof}

\vspace{10pt}

\noindent
\textbf{Remark.}\ 
When $G$ is simply-laced, Proposition~$\ref{prop Delta is nonnegative}$ is a consequence of the positivity of canonical bases (see \cite[Proposition~3.2]{lusztig1994total} and \cite[Sect.\ 1.7]{lusztig1998total} for details).

\vspace{20pt}

Recall that $G_{\ge0}$ was defined by using $x_i$, $y_i$, $\chi$ for $i\in I$ and $\chi\in\Hom(\C^{\times},T)$.
By restricting the range of $i$ to a subset $J\subseteq I$, we define the totally nonnegative parts $(U_J)_{\ge0}$, $(U^-_J)_{\ge0}$, $(G_J)_{\ge0}$ (resp.\ totally positive parts $(U_J)_{>0}$, $(U^-_J)_{>0}$, $(G_J)_{>0}$) in the same ways as above.

\begin{lemma}\label{lem UJ nonnegative}
We have $(U_J)_{\ge0}=U_J\cap U_{\ge0}$.
\end{lemma}

\begin{proof}
The inclusion $(U_J)_{\ge0}\subseteq U_J\cap U_{\ge0}$ is obvious.
To prove the opposite inclusion, suppose that $x\in U_J\cap U_{\ge0}$.
We have $U_J\subseteq G_J$, and $G_J$ is a semisimple group with the Weyl group $W_J$. 
For each $w\in W_J$, we note that $\dot{w}\in G_J$ because of its definition given by \eqref{eq choice of representative 1} and \eqref{eq choice of representative 2}.
Thus, by the Bruhat decomposition of $G_J$ with respect to the opposite Borel subgroup $B^-_J$ of $B_J$, we have
\begin{equation*}
 U_J \subseteq \bigsqcup_{w\in W_J} B_J^- \dot{w} B_J^- 
 \subseteq \bigsqcup_{w\in W_J} B^- \dot{w} B^- .
\end{equation*}
In particular, we have
\begin{equation}\label{eq Bruhat decomposition of UJ}
 U_J\cap B^- \dot{w} B^-=\emptyset \quad \text{for all $w\notin W_J$}.
\end{equation}
Since $x\in U_{\ge0}=\sqcup_{w\in W} U(w)$, there exists $w\in W$ such that $x\in U(w)$. 
Combining this with $U(w)\subseteq B^-wB^-$, we obtain 
\begin{equation*}
 x\in U_J\cap B^-wB^-.
\end{equation*}
This implies by \eqref{eq Bruhat decomposition of UJ} that $w\in W_J$ so that $U(w)\subseteq (U_J)_{\ge0}$.
Since $x\in U(w)$, we conclude $x\in (U_J)_{\ge0}$, as desired.
\end{proof}

\vspace{10pt}

The totally nonnegative part $(G_J/B_J)_{\ge0}$ is also defined similarly.
Note that the inclusion $G_J\hookrightarrow G$ induces a closed embedding
\begin{align*}
 G_J/B_J \hookrightarrow G/B
 \quad ; \quad 
 g_JB_J \mapsto g_JB.
\end{align*}
The following claim seems to be well-known. We give a proof in Appendix (Section~\ref{subsec Appendix 2}).

\begin{lemma}\label{lem closed embedding of TN 3}
Under the closed embedding $G_J/B_J \hookrightarrow G/B$, the image of $(G_J/B_J)_{\ge0}$ is precisely $G_J/B_J\cap (G/B)_{\ge0}$.
\end{lemma}

\vspace{10pt}

Recall from \eqref{eq Schubert cell for wJ} that an element of the Schubert cell $X_{w_J}^{\circ}$ can be written as $x\dot{w}_JB$ for a unique $x\in U_J$.
The next claim gives us a criterion for $x\dot{w}_JB$ to lie in $(G/B)_{\ge 0}$. 

\begin{proposition}\label{prop x in U-ge 0 2}
Let $J\subseteq I$, and suppose that $x\in U_J$. Then, we have
\begin{equation*}
x\dot{w}_JB\in (G/B)_{\ge 0}
\quad \text{if and only if} \quad
x\in (U_J)_{\ge0}.
\end{equation*}
\end{proposition}

\begin{proof}
Note that we have $x\dot{w}_J\in G_J$ because of our choice of representatives given by \eqref{eq choice of representative 1} and \eqref{eq choice of representative 2}. Therefore, it follows that
\begin{align*}
 x\dot{w}_JB\in (G/B)_{\ge 0}
 &\quad \Longleftrightarrow \quad
 x\dot{w}_JB_J\in (G_J/B_J)_{\ge 0} \quad \text{(by Lemma~\ref{lem closed embedding of TN 3})}\\
 &\quad \Longleftrightarrow \quad
 x \in (U_J)_{\ge0} \qquad \text{(by Lemma~\ref{lem yw0=u+b 2} for $G_J$)}.
\end{align*}
\end{proof}

For $K\subseteq J\subseteq I$, we have the associated Richardson stratum $\RY{K}{J}$ in $Y$ (Section~\ref{subsec Richardson strata of Y}).
We set its totally nonnegative part as
\begin{align*}
\RY{K}{J;>0} \coloneqq \RY{K}{J}\cap Y_{\ge0} = \RY{K}{J}\cap (G/B)_{\ge0}.
\end{align*}
It is clear that these pieces partition $Y_{\ge0}$:
\begin{align}\label{eq decomp of TNN of Y}
 Y_{\ge0} = \bigsqcup_{K\subseteq J\subseteq I} \RY{K}{J;>0} .
\end{align}
To give a description of $\RY{K}{J;>0}$, recall from \eqref{eq def of UJeJ} that $(U_J)^{e_J} = U_J\cap Z_{G_J}(e_J)$.
We set
\begin{align*}
&(U_J)^{e_J}_{\ge0} \coloneqq (U_J)_{\ge0}\cap Z_{G_J}(e_J) = (U_J)^{e_J}\cap U_{\ge0},
\end{align*}
where the second equality follows from Lemma~\ref{lem UJ nonnegative}.
The following claim gives us a description of totally nonnegative parts of Richardson strata (cf.\ Remark~\ref{rem on Delta being 1}).

\begin{proposition}\label{prop description for R KI>0 2}
Let $K\subseteq J\subseteq I$. Then, we have 
\begin{equation*}
\RY{K}{J;>0}=\left\{x\dot{w}_JB \ \left| \ 
\begin{matrix}
\text{
\rm $x\in (U_J)^{e_J}_{\ge0}$ and}
~\begin{cases}\Delta_{\varpi_i}(x\dot{w}_J)=0\quad\text{if}\quad i\in K\\
\Delta_{\varpi_i}(x\dot{w}_J)> 0\quad\text{if}\quad i\in J-K \\
\Delta_{\varpi_i}(x\dot{w}_J)=1 \quad\text{if}\quad i\in I-J
\end{cases}
\end{matrix}
\right.
\right\}.
\end{equation*}
\end{proposition}

\begin{proof}
Let $\Ry{K}{J;>0}$ be the right hand side of the desired equality.
We first show that
\begin{align}\label{eq first goal RKI nonnegative 2}
\RY{K}{J;>0} \supseteq \Ry{K}{J;>0}.
\end{align}
Recall that we have $\RY{K}{J;>0}=\RY{K}{J}\cap Y_{\ge0}=\RY{K}{J}\cap (G/B)_{\ge0}$.
It is obvious that $\Ry{K}{J;>0}\subseteq \RY{K}{J}$ by Proposition~\ref{prop description of YKJ}.
Thus, to see the inclusion \eqref{eq first goal RKI nonnegative 2}, it suffices to show that $x\in (U_J)^{e_J}_{\ge0}$ implies that $x\dot{w}_JB\in (G/B)_{\ge0}$. But this follows from Proposition~\ref{prop x in U-ge 0 2}.

To complete the proof, we show that the opposite inclusion holds in what follows.
Since $\RY{K}{J;>0}=\RY{K}{J}\cap (G/B)_{\ge0}$, 
Proposition~\ref{prop description of YKJ} implies that an arbitrary element of $\RY{K}{J;>0}$ can be written as $x\dot{w}_JB$ for some $x\in (U_J)^{e_J}$ such that
\begin{equation*}
\begin{cases}\Delta_{\varpi_i}(x\dot{w}_J)=0\quad\text{if}\quad i\in K\\
\Delta_{\varpi_i}(x\dot{w}_J)\ne 0\quad\text{if}\quad i\in I-K .
\end{cases}
\end{equation*}
Here, we have $\dot{w}_J\in G_J$ from our choice of representatives given by \eqref{eq choice of representative 1} and \eqref{eq choice of representative 2}.
So we additionally have 
\begin{equation*}
\Delta_{\varpi_i}(x\dot{w}_J)=1\quad\text{if}\quad i\in I-J
\end{equation*}
by Proposition~\ref{prop restriction of Delta}~(ii).
Since $x\dot{w}_JB\in\RY{K}{J;>0}=\RY{K}{J}\cap (G/B)_{\ge0}$ and $x\in U_J$, it follows that $x\in (U_J)_{\ge0}\subseteq U_{\ge0}$ by Proposition~\ref{prop x in U-ge 0 2}.
So we have $x\in (U_J)^{e_J}\cap U_{\ge0} = (U_J)^{e_J}_{\ge0}$.
Thus, to see that $x\dot{w}_JB\in \Ry{K}{J;>0}$, it suffices to show that $\Delta_{\varpi_i}(x\dot{w}_J)\ge0$ for all $i\in I$.
But this follows from Proposition~\ref{prop Delta is nonnegative}.
\end{proof}

%\vspace{30pt}

\section{Toric orbifold associated to Cartan matrix}
\label{sec toric}

In this section, we study the toric orbifold which appeared in the introduction. It was first introduced by Blume (\cite{blume2015toric}), and later related topics are studied by several authors (e.g.\ \cite{abe2023peterson,abe2025totally,burrull2023strongly,crowley2024toric,HMSS21-toric-orbifolds}).

\subsection{A fan $\fan$ on $\mathfrak{t}_{\R}$}\label{sec construction of fan}

Recall that $I$ is the Dynkin diagram of $G$ with respect to the maximal torus $T\subseteq G$ and the Borel subgroup $B\subseteq G$ containing $T$.
We begin with constructing a fan on $\mathfrak{t}_{\R}\coloneqq \CWL\otimes_{\Z}\R$, where $\CWL=\Hom(\C^{\times},T/Z)=\bigoplus_{i\in I} \Z \varpi^{\vee}_i$ is the coweight lattice.
We regard $\mathfrak{t}_{\R}$ as a vector space over $\R$ whose lattice of integral vectors is the coweight lattice $\CWL$.

For disjoint subsets $K,J\subseteq I$ (i.e.\ $K\cap J=\emptyset$), let $\sigma_{K,J}\subset \mathfrak{t}_{\R}$ be the cone spanned by the simple coroots $-\alpha^{\vee}_i$ for $i\in K$ and the fundamental coweights $\varpi^{\vee}_i$ for $i\in J$ :
\begin{align*}
 \sigma_{K,J} \coloneqq \text{cone}( \{-\alpha^{\vee}_i \mid i\in K\}\cup\{\varpi^{\vee}_i \mid i\in J \} )\subset \mathfrak{t}_{\R},
\end{align*}
where we take the convention $\sigma_{\emptyset,\emptyset}\coloneqq\{\bm{0}\}$. 

\begin{definition}\label{def fan}
\textnormal{
Let $\fan$ be the set of the cones $\sigma_{K,J}$ for disjoint subsets $K, J\subseteq I$:
\begin{align}\label{eq def of fan}
 \fan \coloneqq\{\sigma_{K,J} \mid K,J\subseteq I, \ K\cap J=\emptyset\}.
\end{align}
}
\end{definition}

\vspace{10pt}

It is known that $\fan$ is a simplicial projective fan on $\mathfrak{t}_{\R}=\CWL\otimes_{\Z}\R$ (\cite{abe2023peterson,blume2015toric})  so that it defines a simplicial projective toric variety. 

\begin{definition}
\textnormal{
We denote by $X(\fan)$ the simplicial projective toric variety 
associated to the fan $\fan$ (on $\mathfrak{t}_{\R}$ with the coweight lattice $\CWL$) defined in \eqref{eq def of fan}.
}
\end{definition}

\vspace{10pt}

By definition, we have $\dim_{\C}X(\fan)=|I|=n$ (which agrees with $\dim_{\C}Y$).
It is worth noting that when we compare the geometry of $X(\fan)$ to that of $Y$, it is convenient to use the following notation; for $K\subseteq J\subseteq I$, we set
\begin{align}\label{eq def of new cone}
 \tau_{K,J} \coloneqq \text{cone}( \{-\alpha^{\vee}_i \mid i\in K\}\cup\{\varpi^{\vee}_i \mid i\notin J \} ) (=\sigma_{K,J^c}).
\end{align}
We then have $\fan = \{\tau_{K,J} \mid K\subseteq J\subseteq I\}$ since the condition $K\subseteq J$ is equivalent to $K\cap J^c=\emptyset$.
In subsequent sections, we review some basic facts from \cite[Sect.~3]{abe2023peterson}.

\vspace{10pt}

\subsection{Homogeneous coordinates of $X(\Sigma)$}

We now give a description of $X(\Sigma)$ in terms of its homogeneous coordinates.
We set $\C^{2I}$ as follows:
\begin{align*}
 \C^{2I} \coloneqq 
 \{ (x_1,\ldots,x_{\rkg};y_1,\ldots,y_{\rkg}) \in \C^{2\rkg} \mid x_i,y_i\in\C \ (i\in I) \}.
\end{align*}
The maximal torus $T$ acts linearly on $\C^{2I}$ through the weights $(\varpi_1,\ldots,\varpi_{\rkg},\alpha_1,\ldots,\alpha_{\rkg})$. 
Namely, we set
\begin{align}\label{eq def of T-action on C2r}
 t\cdot (x_1,\ldots,x_{\rkg};y_1,\ldots,y_{\rkg}) 
 \coloneqq (\varpi_1(t)x_1,\ldots,\varpi_{\rkg}(t)x_{\rkg}; \alpha_1(t)y_1,\ldots,\alpha_{\rkg}(t)y_{\rkg})
\end{align}
for $t\in T$ and $(x_1,\ldots,x_{\rkg};y_1,\ldots,y_{\rkg})\in \C^{2I}$. 
Let $E\subset \C^{2I}$ be a subset whose complement is given by
\begin{align}\label{eq def of C2I - E}
 \C^{2I}-E = \{ (x_1,\ldots,x_{\rkg}; y_1,\ldots,y_{\rkg})\in\C^{2I} \mid  (x_i,y_i)\ne(0,0) \ (i\in I) \}.
\end{align}
It is clear that the linear $T$-action on $\C^{2I}$ defined above preserves the subset $\C^{2I}-E$.
We now consider the quotient space 
\begin{align*}
 (\C^{2I}-E)/T.
\end{align*}
It admits a natural action of the quotient torus $T/Z_G$, where $Z_G$ is the center of $G$. More precisely, the torus $T/Z_G$ acts on $(\C^{2I}-E)/T$ by setting
\begin{align*}
 [t]\cdot [x_1,\ldots,x_{\rkg};y_1,\ldots,y_{\rkg}]
 \coloneqq [x_1,\ldots,x_{\rkg}; \alpha_1(t)y_1,\ldots,\alpha_{\rkg}(t)y_{\rkg}]
\end{align*}
for $[t]\in T/Z_G$ and $[x_1,\ldots,x_{\rkg};y_1,\ldots,y_{\rkg}]\in (\C^{2I}-E)/T$. This action is well-defined since we have $\alpha_i(t)=1$ $(i\in I)$ for all $t\in Z_G$.

In \cite{abe2023peterson}, it was shown that $(\C^{2I}-E)/T$ with the $T/Z_G$-action defined here is the quotient presentation of $X(\Sigma)$. Namely, the $2I$-tuple $(x_1,\ldots,x_{\rkg}; y_1,\ldots,y_{\rkg})$ on $\C^{2I}$ is the homogeneous coordinates of $X(\Sigma)$ (\cite[Chap.~5]{cox2024toric}). 
In this description, the torus invariant irreducible Weil divisors corresponding to the rays generated by $-\alpha^{\vee}_i$ and $\varpi^{\vee}_i$ are given by the equations $x_i=0$ and $y_i=0$, respectively. 

For the rest of this paper, we identify the toric variety $X(\fan)$ and $(\C^{2I}-E)/T$ with this $T/Z_G$-action.

\vspace{10pt}

\subsection{Orbit stratification of $X(\fan)$}

For simplicity, we write elements of $\C^{2I}$ and $X(\fan)(=(\C^{2I}-E)/T)$ by $(x;y)$ and $[x;y]$, respectively.
Having the definition \eqref{eq def of new cone} of the cone $\tau_{K,J}$ in mind, we consider the following subsets of $X(\fan)$.
For $K\subseteq J\subseteq I$, we define $\X{K,J}\subseteq X(\fan)$ by \vspace{5pt}
\begin{equation}\label{eq definition of X sigma IKJ}
\X{K,J}\coloneqq\left\{[x;y]\in X(\fan) \ \left| \ \begin{cases}
x_i=0\quad (i\in K) \\ 
x_i\neq 0\quad (i\in I-K) 
\end{cases}
\hspace{-10pt},\hspace{5pt} 
\begin{cases}
y_i=0\quad (i\in I-J) \\ 
y_i\neq 0\quad (i\in J) 
\end{cases}
\hspace{-5pt}\right\} \right. .\vspace{5pt}
\end{equation}
Note that the right hand side is the empty set unless $K\subseteq J$ because of the definition of $E$ (see \eqref{eq def of C2I - E}). 
Each $\X{K,J}$ is the torus orbit of $X(\Sigma)$ corresponding to the cone $\tau_{K,J}$. This follows from the correspondence between the equation $x_i=0$ (resp.\ $y_i=0$) and the ray generated by $-\alpha^{\vee}_i$ (resp. $\varpi^{\vee}_i$). For a detail argument, see \cite[Sect.~3.3]{abe2025totally}. Thus, the torus orbit stratification of $X(\fan)$ is given by
\begin{align}\label{eq X oribt decomp}
X(\fan) = \bigsqcup_{K\subseteq J \subseteq I} \X{K,J} .
\end{align}

To give an effective description of $\X{K,J}$, let us recall some properties of the subtorus $T_J\subseteq T$ defined in Section~\ref{subsec Richardson strata of Y}.
We have an isomorphism
\begin{align}\label{eq alpha J surj 2}
 T_J \rightarrow (\C^{\times})^J
 \quad ; \quad
 t\mapsto (\varpi_i(t))_{i\in J}
\end{align} 
and a surjective homomorphism
\begin{align}\label{eq alpha J surj}
 T_{J}\rightarrow (\C^{\times})^{J}
 \quad ; \quad
 t\mapsto (\alpha_i(t))_{i\in J}.
\end{align} 
We also know that
\begin{align}\label{eq alpha J being 1}
 \varpi_i(T_{J}) = 1 \quad \text{for all $i\in I-J$}
\end{align}
since $T_J$ is generated by $\alpha^{\vee}_i(\C^{\times})$ for $i\in J$ (Section~\ref{subsec Richardson strata of Y}).

\begin{proposition}\label{prop torus orbit stratum}
For $K\subseteq J\subseteq I$, we have 
\begin{equation*}
\X{K,J}=\left\{[x;y] \in X(\fan) 
\ \left| \ \begin{cases}
x_i=0\quad (i\in K) \\ 
x_i\ne0\quad (i\in J-K) \\ 
x_i=1\quad (i\in I-J) 
\end{cases}
\hspace{-10pt},\hspace{5pt} 
\begin{cases}
y_i=0\quad (i\in I-J) \\
y_i=1\quad (i\in J) 
\end{cases} \hspace{-10pt} \right\} \right. .
\end{equation*}
\end{proposition}

\begin{proof}
The inclusion $\supseteq$ is obvious due to \eqref{eq definition of X sigma IKJ}.
To show that the opposite inclusion holds, recall that each element $[x;y]$ of $X(\Sigma)$ is a $T$-orbit in $\C^{2I}-E$ for the action given by the weights $\varpi_1,\ldots,\varpi_{\rkg},\alpha_1,\ldots,\alpha_{\rkg}$ (see \eqref{eq def of T-action on C2r}), and $(x;y)$ is a representative of that $T$-orbit.
Now, let $[x;y]$ be an element of $\X{K,J}$. 
Then, by \eqref{eq definition of X sigma IKJ}, we have $x_i\ne0$ for $i\in I-K$.
Note that this index set contains $I-J$.
Hence, using \eqref{eq alpha J surj 2} for $T_{I-J}\subseteq T$, we may assume that
\begin{align}\label{eq xi = 1}
 x_i = 1 \quad \text{for all $i\in I-J$}.
\end{align}
Moreover, the surjectivity of \eqref{eq alpha J surj} for $T_J\subseteq T$ implies that we may additionally assume that $y_i=1$ for all $i\in J$ while keeping \eqref{eq xi = 1} because of the equalities \eqref{eq alpha J being 1}.
This implies that a representative $(x;y)$ of the given $[x;y]$ can be chosen as indicated in the right hand side of the desired equality. This proves the opposite inclusion.
\end{proof}

This result is compatible with the fact $\codim_{\R}\tau_{K,J}=|J|-|K|$ (\cite[Sect.~3]{abe2025totally}).

\vspace{10pt}

\subsection{Nonnegative part of $X(\fan)$}\label{sec nonnegative part of X}

In general, for an arbitrary (normal) toric variety $X$ over $\C$, the nonnegative part $X_{\ge 0}$ is defined as the set of ``$\R_{\ge 0}$-valued points" of $X$. When $X$ does not have torus factors, it can be obtained by requiring nonnegative entries in the homogeneous coordinates of $X$ (\cite[Proposition 12.2.1]{cox2024toric}). More details on $X_{\ge 0}$ can be found in Section~12.2 of loc.\ cit.
In this section, we study the nonnegative part of the toric orbifold $X(\fan)$ defined in the previous section. 

Since $X(\Sigma)$ is projective, it does not have torus factors, and hence its nonnegative part is given by
\begin{align*}
X(\fan)_{\ge0} = \{[x;y]\in X(\fan) \mid x_i, y_i\ge0\ \text{for $i\in I$}\}.
\end{align*}
Namely, an element of $X(\fan)$ belongs to $X(\fan)_{\ge0}$ if and only if it can be represented by an element with nonnegative entries in the homogeneous coordinate.

To obtain a stratification of $X(\fan)_{\ge0}$ from the orbit stratification, we set 
\begin{align*}
\Xp{K,J} \coloneqq \X{K,J}\cap X(\fan)_{\ge 0}
\quad \text{for $K\subseteq J \subseteq I$.}
\end{align*}
We call each of them a \textit{nonnegative torus orbit} of $X(\Sigma)_{\ge0}$.
We obtain from \eqref{eq X oribt decomp} that
\begin{align}\label{eq decomp of X nonnegative KJ}
X(\fan)_{\ge0} = \bigsqcup_{K\subseteq J \subseteq I} \Xp{K,J} .
\end{align}

To give a description of $\Xp{K,J}$ in the same manner as Proposition~\ref{prop torus orbit stratum}, we need the following lemma.
Let $(T_J)_{>0}$ be the subgroup of $T_J$ generated by $\chi(z)$ for $\chi\in\Hom(\C^{\times},T_J)$ and $z\in\R_{>0}$ (cf.\ Section~\ref{TNN G and G/B}).
Since we have $\Hom(\C^{\times},T_J)=\oplus_{i\in J}\Z\alpha^{\vee}_i$, it follows that $(T_J)_{>0}$ is generated by $\alpha^{\vee}_i(z)$ for $i\in I$ and $z\in\R_{>0}$.

\begin{lemma}\label{eq alpha surjectivity}
For $t\in (T_J)_{>0}$, we have $\varpi_i(t),\alpha_i(t)\in\R_{>0}$ for all $i\in I$.
Moreover, we have isomorphisms
\begin{align*}
 &(T_J)_{>0} \rightarrow (\R_{>0})^J
 \quad ; \quad
 t\mapsto (\varpi_i(t))_{i\in J} ,\\
 &(T_{J})_{>0}\rightarrow (\R_{>0})^{J}
 \quad ; \quad
 t\mapsto (\alpha_i(t))_{i\in J}.
\end{align*} 
\end{lemma}

\vspace{10pt}

These maps are totally positive analogues of \eqref{eq alpha J surj 2} and \eqref{eq alpha J surj}.
The first claim of this lemma simply follows since $(T_J)_{>0}$ is generated by $\alpha^{\vee}_i(z)$ $i\in I$ and $z\in\R_{>0}$ as we saw above.
We give a proof of the second claim in Appendix.

\begin{proposition}\label{prop nonnengative stratum}
For $K\subseteq J\subseteq I$, we have 
\begin{equation*}
\Xp{K,J}=\left\{[x;y] \in X(\fan) 
\ \left| \ \begin{cases}
x_i=0\quad (i\in K) \\ 
x_i>0\quad (i\in J-K) \\ 
x_i=1\quad (i\in I-J) 
\end{cases}
\hspace{-10pt},\hspace{5pt} 
\begin{cases}
y_i=0\quad (i\in I-J) \\
y_i=1\quad (i\in J) 
\end{cases} \hspace{-10pt} \right\} \right. .
\end{equation*}
\end{proposition}

\begin{proof}
The inclusion $\supseteq$ is obvious.
For the opposite inclusion, we take an element $[x;y]\in \Xp{K,J}=\X{K,J}\cap X(\fan)_{\ge 0}$ of the form given in \eqref{eq definition of X sigma IKJ}.
Since $[x;y]\in X(\fan)_{\ge 0}$, we have $[x;y]=[x';y']$ for some $(x';y')\in \C^{2I}-E$ with nonnegative entries.
Because of the definition of our equivalence class, we have
\begin{equation*}
 x_i\ne0 \ \Longleftrightarrow \ x'_i\ne0
 \qquad \text{and} \qquad
 y_i\ne0 \ \Longleftrightarrow \ y'_i\ne0.
\end{equation*}
Since we have Lemma~\ref{eq alpha surjectivity} (instead of \eqref{eq alpha J surj 2} and \eqref{eq alpha J surj}), the proof of Proposition~\ref{prop torus orbit stratum} applying to $[x';y']$ works to obtain the desired inclusion by replacing the roles of $T_{I-J}$ and $T_J$ to those of $(T_{I-J})_{>0}$ and $(T_J)_{>0}$, respectively.
\end{proof}

\vspace{20pt}

Motivated by Proposition~\ref{prop nonnengative stratum}, we consider the following definition; 
for $K\subseteq J\subseteq I$, we set
\begin{equation*}
(\R^{I})_{K,J;>0} \coloneqq
\left\{(x_1,\ldots,x_{\rkg}) \in \R^{I} 
\ \left| \ \begin{array}{l}
x_i=0\quad (i\in K) \\ 
x_i>0\quad (i\in J-K) \\
x_i=1\quad (i\in I-J) 
\end{array}
 \right\} \right. .
\end{equation*}
The next claim gives us a parametrization of the nonnegative part $\Xp{K,J}$.

\begin{proposition}\label{prop nonnengative stratum parametrization}
For $K\subseteq J\subseteq I$, the map
\begin{align*}
 (\R^{I})_{K,J;>0} \rightarrow \Xp{K,J}
 \quad ; \quad 
 (x_1,\ldots,x_{\rkg}) \mapsto [x_1,\ldots,x_{\rkg}\hspace{2pt};\delta_1,\ldots,\delta_{\rkg}]
\end{align*} 
is a bijection, where $(\delta_1,\ldots,\delta_{\rkg})\in\R^{I}$ is the element given by
\begin{align*}
\begin{cases}
 \delta_i=0 \quad (i\in I-J) \\
 \delta_i=1 \quad (i\in J) .
\end{cases}
\end{align*}
\end{proposition}

\begin{proof}
The surjectivity is obvious by the previous proposition.
We prove the injectivity. Suppose that $(x_1,\ldots,x_{\rkg}), (x'_1,\ldots,x'_{\rkg})\in (\R^{I})_{K,J;>0}$ have the same image in $\Xp{K,J}$. Then there exists $t\in T$ such that
\begin{equation}\label{eq connecting t in T}
\begin{split}
 &x_i = \varpi_i(t) \cdot x'_i \quad (i\in I), \\
 &1 = \alpha_j(t) \cdot 1 \quad (j\in J).
\end{split}
\end{equation}
From the former condition, we obtain that $\varpi_i(t)=1$ for $i\in I-J$ since $x_i=x'_i=1$ for all such $i$.
This means that we have $t\in \cap_{i\in I-J}\Ker \varpi_i = T_J$, where the equality follows since $T_J$ is a $|J|$-dimensional torus generated by $\alpha^{\vee}_j(\C^{\times})\subseteq T$ for $j\in J$ (Section~\ref{subsec Richardson strata of Y}). This and the latter condition in \eqref{eq connecting t in T} together imply that $t\in Z_{G_J}$. 
Since $G_J$ is semisimple, its center $Z_{G_J}$ is a finite group. Hence, $t$ has a finite order, and so are the values $\varpi_i(t)$ in $\C^{\times}$. Therefore, we have $|\varpi_i(t)|=1$ for $i\in I$.
Now, the former condition in \eqref{eq connecting t in T} with $x_i,x'_i\ge0$ imply that
\begin{align*}
 x_i = x'_i \quad (i\in I).
\end{align*}
In fact, if $x_i=0$ then we have $x_i=x'_i=0$, and if $x_i\ne0$ then both of $x_i$ and $x'_i$ are positive which implies that $\varpi_i(t)=1$ and hence $x_i = x'_i$. This completes the proof.
\end{proof}

\vspace{10pt}

\subsection{Strongly dominant weight polytope}\label{subsec SDWP}

Let $\lambda\in\WL=\Hom(T,\C^{\times})$ be a regular dominant weight. Namely, we have $\lambda=\sum_{i\in I} a_i \varpi_i$ for some positive coefficients $a_i\in\R_{>0}$ $(i\in I)$.
Regarding $\lambda$ as an element of $\mathfrak{t}^*_{\R}\coloneqq \WL\otimes_{\Z}\R$, the weight polytope associated to $\lambda$ is the convex hull of the $W$-orbit of $\lambda$ :
\begin{align*}
 \text{Conv}(W\cdot \lambda) \subseteq \mathfrak{t}^*_{\R}.
\end{align*}
For its relations to the irreducible representation $V_{\lambda}$ with highest weight $\lambda$, see e.g.\ \cite[Sect.~10.1]{hall2015gtm222}.
According to \cite{burrull2023strongly}, we define the \textit{strongly dominant weight polytope} $P^{\lambda}$ by the intersection
\begin{align*}
 P^{\lambda} \coloneqq \text{Conv}(W\cdot \lambda) \cap \sigma_+ \subseteq \mathfrak{t}^*_{\R},
\end{align*}
where
$
 \sigma_+ \coloneqq \{ a_1 \varpi_1 + \cdots + a_{\rkg} \varpi_{\rkg} \mid a_i\ge0 \ (i\in I) \}
$
is the (closed) dominant Weyl chamber in $\mathfrak{t}^*_{\R}$.
In loc.\ cit., $P^{\lambda}$ is proved to be a rational\footnote{This means that the vertices of $P^{\lambda}$ lie in $\WL\otimes_{\Z}\Q$.} combinatorial cube of dimension $n(=|I|)$. Topologically, this is equivalent to the existence of a (piecewise linear) homeomorphism between $P^\lambda$ and the standard $n(=|I|)$-cube, which restricts to homeomorphisms between their facets (and hence all the faces), e.g.\ \cite[Sect.\ 2.2]{ziegler1995lectures}. In particular, its combinatorial structure is independent of the choice of the regular dominant weight $\lambda$.
We review some facts about $P^{\lambda}$ from loc.\ cit.
We set 
\begin{align*}
 &H_i \coloneqq \{ \mu \in \mathfrak{t}^*_{\R} \mid \alpha^{\vee}_i(\mu)= 0 \}, \\ 
 &H^{\lambda}_i \coloneqq \{ \lambda - \mu \in \mathfrak{t}^*_{\R} \mid \varpi^{\vee}_i(\mu)= 0 \}.
\end{align*}
Then the set of facets of $P^{\lambda}$ are given by $P^{\lambda}\cap H_i$ and $P^{\lambda}\cap H^{\lambda}_i$ for $i\in I$ with outward normal vectors $-\alpha^{\vee}_i$ $(i\in I)$ and $\varpi^{\vee}_i$ $(i\in I)$, respectively.
Moreover, for $K, J\subseteq I$, the intersection
\begin{align}\label{eq faces of SDWP}
 F_{K,J} \coloneqq P^{\lambda}\cap \bigcap_{j\in K} H_i \cap \bigcap_{j\notin J} H^{\lambda}_i
\end{align}
is non-empty if and only if $K\subseteq J$ (equivalently, $K\cap (I-J)=\emptyset$). 
See \cite[Sect.~2]{burrull2023strongly} for details.
This means that the normal fan of $P^{\lambda}$ is precisely our fan $\Sigma$ given in Definition~\ref{def fan}.

Since $\Sigma$ is the normal fan of $P^{\lambda}$ in $\mathfrak{t}^*_{\R}$, we have the moment map
\begin{align*}
 \mu \colon X(\Sigma) \rightarrow \mathfrak{t}^*_{\R},
\end{align*}
which restricts to a homeomorphism
\begin{align}\label{eq moment map on NN}
 \mu_{\ge0} \colon X(\Sigma)_{\ge0} \rightarrow P^{\lambda}.
\end{align}
The latter map sends each nonnegative torus orbit $\Xp{K,J}$ to the relative interior of the corresponding face $F_{K,J}$ of $P^{\lambda}$. See \cite[Sect.~4.2]{fulton1993introduction} for details.
In this sense, the nonnegative part $X(\Sigma)_{\ge0}$ is naturally identified with the strongly dominant weight polytope $P^{\lambda}$.

\vspace{30pt}

\section{The map $\Psi_{\ge0} \colon Y_{\ge0}\rightarrow X(\Sigma)_{\ge0}$} \label{sec map}

In \cite{abe2023peterson}, the first and the third authors of this paper constructed a distinguished morphism  $\Psi\colon Y\rightarrow X(\Sigma)$. We use this map to establish a connection between $Y_{\ge0}$ and $X(\Sigma)_{\ge0}$.

\subsection{Construction of $\Psi \colon Y\rightarrow X(\Sigma)$}

In this section, we review the construction of the morphism $\Psi\colon Y\rightarrow X(\Sigma)$ from \cite{abe2023peterson}.
Let $i\in I$. Recall that we have
\begin{align*}
 \Delta_{\varpi_i} \colon G \rightarrow \C
 \quad ; \quad
 g \mapsto (gv_{\varpi_i})_{\varpi_i} ,
\end{align*} 
where $(gv_{\varpi_i})_{\varpi_i}$ denotes the coefficient of $v_{\varpi_i}$ for the weight decomposition of $gv_{\varpi_i}$ in the fundamental representation $V_{\varpi_i}$.
Let us consider another function
\begin{align*}
 \qp{i} \colon G \rightarrow \C
 \quad ; \quad
 g \mapsto - (\Ad_{g^{-1}}\nil)_{-\alpha_i},
\end{align*}
where we have $\nil=\sum_{i\in I}\nil_i$, and $(\Ad_{g^{-1}}\nil)_{-\alpha_i}$ denotes the coefficient of the root vector $f_i\in \mathfrak{g}_{-\alpha_i}$ for the root decomposition of the element $\Ad_{g^{-1}}\nil\in\mathfrak{g}$. 

\begin{definition}
We denote by $\Psi$ the map from $Y$ to $X(\Sigma)$ defined by
\begin{equation*}
 \Psi\colon Y\rightarrow X(\Sigma)
 \quad ; \quad
 gB \mapsto [\Delta_{\varpi_1}(g),\ldots,\Delta_{\varpi_{\rkg}}(g);\qp{1}(g),\ldots,\qp{n}(g)].
\end{equation*}
\end{definition}

\vspace{5pt}

The map $\Psi$ is well-defined essentially because of the $T$-action given in \eqref{eq def of T-action on C2r} to construct $X(\Sigma)=(\C^{2I}-E)/T$. Moreover, $\Psi$ is a morphism of algebraic varieties. See \cite[Sect.\ 6.1]{abe2023peterson} for detail proofs.

\vspace{10pt}

\subsection{Restriction of $\Psi$ to $Y_{\ge0}$}

The goal of this section is to prove that $\Psi$ restricts to a map 
\begin{equation*}
 Y_{\ge0}\rightarrow X(\Sigma)_{\ge0}.
\end{equation*}
We begin with the following lemma. Recall that we have chosen a representative $\dot{w}\in N_G(T)$ given in \eqref{eq choice of representative 2} for each $w\in W$.

\begin{lemma}\label{lemm w0e=-f first}
For $i\in I$, we have 
\begin{equation*}
 \Ad_{\dot{w}_0^{-1}} e_i = - f_{i^*},
\end{equation*}
where $*\colon I \rightarrow I$ is the involution defined by the condition $w_0\alpha_i=-\alpha_{i^*}$ for all $i\in I$.
\end{lemma}

\begin{proof}
A quick proof can be found in \cite[Lemma 5.1]{rietsch2008mirror}.
It uses a slightly different representative for $s_i$ $(i\in I)$, but the argument works for our representative as well.
\end{proof}

\vspace{5pt}

\begin{lemma}\label{lemm ad part for xwJB-}
Let $J\subseteq I$.
If $x\in U_J$, then 
\begin{equation*}
\qp{i}(x\dot{w}_J)=-\big(\Ad_{(x\dot{w}_J)^{-1}} \nil \big)_{-\alpha_i}=\begin{cases}
0\quad\text{if\quad$i\in I-J$}\\
1\hspace{4mm}\text{if\quad$i\in J$}.
\end{cases}
\end{equation*}
\end{lemma}

\begin{proof}
We first consider the case $i\in I-J$. Let $P_J$ be the standard parabolic subgroup of $G$ associated to the subset $J\subseteq I$ satisfying $B\subseteq P_J$.
Since $w_J\in W_J$, we have $\dot{w}_J\in G_J$ because of our choice of representatives given by \eqref{eq choice of representative 1} and \eqref{eq choice of representative 2}.
So we have $x\dot{w}_J\in G_J\subseteq P_J$ and $e\in \mathfrak{u}\subseteq \mathfrak{p}_J$ which imply that $\Ad_{(x\dot{w}_J)^{-1}}e\in \mathfrak{p}_J$. Hence the condition $i\notin J$ implies that we have
\begin{equation*}
 \big(\Ad_{(x\dot{w}_J)^{-1}} e\big)_{-\alpha_i}
  =0.
\end{equation*}

We next consider the case $i\in J$.
Let us write 
\begin{equation*}
 e = e_J + e_{J^c},
\end{equation*}
where we have $e_J=\sum_{j\in J} e_{j}$ and $e_{J^c}=\sum_{j\notin J} e_{j}$.
Then 
\begin{equation*}
 \big(\Ad_{(x\dot{w}_J)^{-1}} e\big)_{-\alpha_i} 
 = \big(\Ad_{(x\dot{w}_J)^{-1}} e_J\big)_{-\alpha_i}   + \big(\Ad_{(x\dot{w}_J)^{-1}} e_{J^c}\big)_{-\alpha_i} .
\end{equation*}
Let $V_J\subseteq U$ be the unipotent radical of $P_J$ so that we have the semidirect decomposition $P_J=V_JL_J$.
We denote by $\mathfrak{v}_J$ the Lie algebra of $V_J$. 
Then we have $e_{J^c}\in\mathfrak{v}_J$.
Noticing that $x\dot{w}_J\in G_J \subseteq P_J$, it follows that $\Ad_{(x\dot{w}_J)^{-1}} e_{J^c}\in \mathfrak{v}_J$ since $\mathfrak{v}_J$ is preserved under the adjoint action of $P_J$.
Recalling that $\mathfrak{v}_J\subseteq \mathfrak{u}$, this implies that the second summand in the above equality vanishes. Hence, we have
\begin{align*}
 \big(\Ad_{(x\dot{w}_J)^{-1}} e\big)_{-\alpha_i} 
 = \big(\Ad_{\dot{w}_J^{-1}}(\Ad_{x^{-1}} e_J)\big)_{-\alpha_i} .
\end{align*}
Since $x\in U_J$, we may have $\Ad_{x^{-1}} e_J = e_J + X_J$ for some $X_J\in\mathfrak{u}_J$ consisting of root vectors for roots in $\Phi_J^+$ with height $\ge2$.
Thus,
\begin{align*}
 \big(\Ad_{\dot{w}_J^{-1}}(\Ad_{x^{-1}} e_J)\big)_{-\alpha_i} 
 &= \big(\Ad_{\dot{w}_J^{-1}}(e_J+X_J)\big)_{-\alpha_i} \\
 &= \big(\Ad_{\dot{w}_J^{-1}}e_J\big)_{-\alpha_i}+\big(\Ad_{\dot{w}_J^{-1}}X_J\big)_{-\alpha_i} \\
 &=-1+0 ,
\end{align*}
where the last equality follows from Lemma~\ref{lemm w0e=-f first} for $G_J$ and its Weyl group $W_J$.
\end{proof}

\vspace{7pt}

\begin{proposition}\label{prop restrict}
The morphism $\Psi \colon Y\to X(\fan)$ restricts to a continuous map
\begin{equation*}
\Psi_{\ge 0} \colon Y_{\ge 0}\to X(\fan)_{\ge 0}
\end{equation*}
which sends $\RY{K}{J;>0}$ to $\Xp{K,J}$ for $K\subseteq J \subseteq  I$.
\end{proposition}

\begin{proof}
By the stratifications of $Y_{\ge 0}$ and $X(\fan)_{\ge 0}$ given in \eqref{eq decomp of TNN of Y} and \eqref{eq decomp of X nonnegative KJ}, respectively, it suffices to show that
\begin{equation*}
\Psi(\RY{K}{J;>0})\subseteq \Xp{K,J}
\end{equation*}
for $K\subseteq J\subseteq I$.
By Proposition~\ref{prop description for R KI>0 2}, an arbitrary element of $\RY{K}{J;>0}$ can be written as $x\dot{w}_JB$ for some $x\in (U_J)^{e_J}_{\ge0}$ satisfying
\begin{equation*}
\begin{cases}
\Delta_{\varpi_i}(x\dot{w}_J)=0\quad\text{if}\quad i\in K\\
\Delta_{\varpi_i}(x\dot{w}_J)>0 \quad\text{if}\quad i\in J-K \\
\Delta_{\varpi_i}(x\dot{w}_J)=1 \quad\text{if}\quad i\in I-J  .
\end{cases}
\end{equation*}
Since $x\in U_J$, we have from Lemma~\ref{lemm ad part for xwJB-} that
\begin{equation*}
\qp{i}(x\dot{w}_J)
=\begin{cases}
0\quad\text{if $i\in I-J$}\\
1\hspace{4mm}\text{if $i\in J$} .
\end{cases}
\end{equation*}
Thus, $\Psi(x\dot{w}_JB)$ belongs to $\Xp{K,J}$ by Proposition~\ref{prop nonnengative stratum}, as desired.
\end{proof}

\vspace{10pt}

The next goal is to show that $\Psi_{\ge 0} \colon Y_{\ge 0}\to X(\fan)_{\ge 0}$ is a homeomorphism.
Since both spaces are compact Hausdorff spaces, it suffices to show that the map is bijective.

\vspace{10pt}

\subsection{Reduction to the key claim}\label{subsec Reduction to I}

As we saw in above, we have a restricted map
\begin{equation}\label{eq restricted Psi TNN}
\Psi_{\ge 0} \colon \RY{K}{J;>0} \rightarrow \Xp{K,J}
\end{equation}
for each $K\subseteq J\subseteq I$.
To show that the whole map $\Psi_{\ge 0} \colon Y_{\ge 0}\to X(\fan)_{\ge 0}$ is a bijection (and hence a homeomorphism), it is enough to prove that \eqref{eq restricted Psi TNN} is a bijection for all $K\subseteq J\subseteq I$.
Our strategy is to apply the following proposition. 
Recall that $\Delta_{\varpi_i}(x\dot{w}_0)\ge0$ $(i\in I)$ for $x\in U^e_{\ge0}$
by Proposition~\ref{prop Delta is nonnegative}.

\begin{theorem}\label{thm LR in our setting}
The map 
\begin{align*}
 U^e_{\ge0} \rightarrow \R^{I}_{\ge0}
 \quad ; \quad
 x \mapsto 
 (\Delta_{\varpi_1}(x\dot{w}_0), \ldots, \Delta_{\varpi_{\rkg}}(x\dot{w}_0))
\end{align*}
is a homeomorphism.
 \end{theorem}

This is essentially a result of Lam--Rietsch (\cite[Theorem~7.3]{lam2016total}),  but their theorem asserts that the claim holds for a \textit{simple} algebraic group of \textit{adjoint type} with slightly modified functions. In subsequent sections, we deduce Theorem~\ref{thm LR in our setting} from loc.\ cit.
Before giving a detail proof, we here apply this claim to show that the map \eqref{eq restricted Psi TNN} is a bijection. \\

\noindent
\textbf{Injectivity of the map \eqref{eq restricted Psi TNN}:}
Suppose that $gB, g'B\in \RY{K}{J;>0}$ satisfy $\Psi(gB)=\Psi(g'B)$.
By Proposition~\ref{prop description for R KI>0 2}, we may assume that
\begin{align*}
 g = x\dot{w}_J, \quad g' = x'\dot{w}_J
 \quad \text{for some $x,x'\in (U_J)^{e_J}_{\ge0}$}
\end{align*}
which satisfy 
\begin{equation*}
\begin{cases}
\Delta_{\varpi_i}(x\dot{w}_J)=0\quad (i\in K)\\
\Delta_{\varpi_i}(x\dot{w}_J)>0 \quad (i\in J-K) \\
\Delta_{\varpi_i}(x\dot{w}_J)=1 \quad (i\in I-J)
\end{cases}
\quad \text{and} \qquad
\begin{cases}
\Delta_{\varpi_i}(x'\dot{w}_J)=0\quad (i\in K)\\
\Delta_{\varpi_i}(x'\dot{w}_J)>0 \quad (i\in J-K) \\
\Delta_{\varpi_i}(x'\dot{w}_J)=1 \quad (i\in I-J).
\end{cases}
\end{equation*}
We also have 
\begin{equation*}
\begin{cases}
\qp{i}(x\dot{w}_J)=0\quad\text{if\quad$i\in I-J$}\\
\qp{i}(x\dot{w}_J)=1\hspace{4mm}\text{if\quad$i\in J$}
\end{cases}
\quad \text{and} \qquad
\begin{cases}
\qp{i}(x'\dot{w}_J)=0\quad\text{if\quad$i\in I-J$}\\
\qp{i}(x'\dot{w}_J)=1\hspace{4mm}\text{if\quad$i\in J$}
\end{cases}
\end{equation*}
by Lemma~\ref{lemm ad part for xwJB-}.
Hence, the assumption $\Psi(gB)=\Psi(g'B)$ and the injectivity of the parametrization of $\Xp{K,J}$ (Proposition~\ref{prop nonnengative stratum parametrization}) imply that 
\begin{equation*}
(\Delta_{\varpi_1}(x\dot{w}_J),\ldots,\Delta_{\varpi_{\rkg}}(x\dot{w}_J) )
=
(\Delta_{\varpi_1}(x'\dot{w}_J),\ldots,\Delta_{\varpi_{\rkg}}(x'\dot{w}_J) ).
\end{equation*}
In particular, we have  
\begin{equation*}
 \Delta_{\varpi_j}(x\dot{w}_J) = \Delta_{\varpi_j}(x'\dot{w}_J)
\qquad (j\in J).
\end{equation*}
Noticing that we have $x\in U_J\subseteq G_J$ and $\dot{w}_J\in G_J$ by \eqref{eq choice of representative 2}, we can express this equality as
\begin{equation*}
 \Delta_{\varpi'_j}(x\dot{w}_J) = \Delta_{\varpi'_j}(x'\dot{w}_J)
\qquad (j\in J)
\end{equation*}
by Proposition~\ref{prop restriction of Delta}~(i), where $\Delta_{\varpi'_j}$ is the function on $G_J$ defined in \eqref{eq def of Delta'}.
Now, applying Theorem~\ref{thm LR in our setting} to $(U_J)^{e_J}_{\ge0}$ in the (simply connected) semisimple group $G_J$, we obtain
\begin{equation*}
 x = x' \quad \text{in $U_J$}.
\end{equation*}
In particular, we obtain that $x\dot{w}_J B = x'\dot{w}_J B$ in $\RY{K}{J;>0}$ . This proves the injectivity of the map \eqref{eq restricted Psi TNN}.\qed\\

\noindent
\textbf{Surjectivity of the map \eqref{eq restricted Psi TNN}:}
Let $[x;y]\in \Xp{K,J}$ be an arbitrary element. By the description of $\Xp{K,J}$ given in Proposition~\ref{prop nonnengative stratum}, we may assume that
\begin{equation}\label{eq Surjectivity of Psi 10}
\begin{split}
\begin{cases}
x_i=0\quad (i\in K) \\ 
x_i>0\quad (i\in J-K) \\ 
x_i=1\quad (i\in I-J) 
\end{cases}
\quad \text{and} \quad 
\begin{cases}
y_i=0\quad (i\in I-J) \\
y_i=1\quad (i\in J). 
\end{cases} 
\end{split}
\end{equation}
Applying Theorem~\ref{thm LR in our setting} to $(U_J)^{e_J}_{\ge0}$ in $G_J$, it follows that there exists $x\in (U_J)^{e_J}_{\ge0}$ with the representative $\dot{w}_J\in G_J$ (chosen as in \eqref{eq choice of representative 1} and \eqref{eq choice of representative 2} for the semisimple group $G_J$) which satisfies
\begin{equation}\label{eq surjectivity Delta' xwJ}
\Delta_{\varpi'_i}(x\dot{w}_J)=x_i\quad (i\in J),
\end{equation}
where $\Delta_{\varpi'_i}$ is the function on $G_J$ defined in \eqref{eq def of Delta'}.
Since $x\dot{w}_J\in G_J$, we can replace $\Delta_{\varpi'_i}$ in \eqref{eq surjectivity Delta' xwJ} to $\Delta_{\varpi_i}$ by Proposition~\ref{prop restriction of Delta}~(i).
We also have 
\begin{equation*}
\Delta_{\varpi_i}(x\dot{w}_J)=1=x_i \quad (i\in I-J) 
\end{equation*}
by Proposition~\ref{prop restriction of Delta}~(ii).
Combining these equalities, we obtain that
\begin{equation*}
\Delta_{\varpi_i}(x\dot{w}_J)=x_i \quad (i\in I) .
\end{equation*}
This and \eqref{eq Surjectivity of Psi 10} together with $x\in (U_J)^{e_J}_{\ge0}$ imply that $x\dot{w}_JB\in \RY{K}{J;>0}$ by Proposition~\ref{prop description for R KI>0 2}.
In addition, we have 
\begin{equation*}
\begin{cases}
\qp{i}(x\dot{w}_J)=0=y_i\quad (i\in I-J) \\
\qp{i}(x\dot{w}_J)=1=y_i\quad (i\in J) 
\end{cases} 
\end{equation*}
by Lemma~\ref{lemm ad part for xwJB-} since $x\in U_J$. Namely, we have $\qp{i}(x\dot{w}_J)=y_i$ for all $i\in I$.
Hence, we obtain that $\Psi(x\dot{w}_JB)=[x;y]$ which proves the surjectivity of the map \eqref{eq restricted Psi TNN}.\qed

\vspace{20pt}

Now we proved that the map \eqref{eq restricted Psi TNN} is bijective for all $K\subseteq J\subseteq I$.
This means that the whole map 
$\Psi_{\ge 0} \colon Y_{\ge 0}\to X(\fan)_{\ge 0}$ 
is a bijection, and hence a homeomorphism (as we pointed out at the end of the previous section).
Therefore, we obtain the following theorem based on Theorem~\ref{thm LR in our setting} (which we prove in subsequent sections).

\begin{theorem} \label{thm homeo}
The map 
\begin{equation*}
 \Psi_{\ge0}\colon Y_{\ge 0}\rightarrow X(\Sigma)_{\ge0}
 \quad ; \quad
 gB \mapsto [\Delta_{\varpi_1}(g),\ldots,\Delta_{\varpi_{\rkg}}(g);\qp{1}(g),\ldots,\qp{n}(g)]
\end{equation*}
is a homeomorphism satisfying
\begin{equation*}
\Psi_{\ge 0}(\RY{K}{J;>0}) = \Xp{K,J}
\end{equation*}
for $K\subseteq J \subseteq I$.
\end{theorem}

\vspace{10pt}

Finally, recall that we have the strongly dominant weight polytope $P^{\lambda}$ for a regular dominant weight $\lambda$ (Section~\ref{subsec SDWP}).
Its faces are $F_{K,J}\subseteq P^{\lambda}$ given in \eqref{eq faces of SDWP} for $K\subseteq J \subseteq I$.
Combining the homeomorphism $\Psi_{\ge0}$ above with the moment map $\mu_{\ge0} \colon X(\Sigma)_{\ge0} \rightarrow P^{\lambda}$ given in \eqref{eq moment map on NN}, we obtain the desired connection between $Y_{\ge 0}$ and $P^{\lambda}$ as follows.
\begin{theorem}
The composition $\mu_{\ge0} \circ \Psi_{\ge0}$ is a homeomorphism
\begin{equation*}
 Y_{\ge 0}\rightarrow P^{\lambda}
\end{equation*}
which sends $\RY{K}{J;>0}$ to $F_{K,J}$
for $K\subseteq J \subseteq I$.
\end{theorem}

\vspace{10pt}

Now our goal is to give a proof of Theorem~\ref{thm LR in our setting} for an arbitrary simply connected semisimple algebraic group $G$ over $\C$ (to apply it to $G_J$ which is not simple in most cases even when $G$ is simple). We deduce it from Lam--Rietsch's result \cite[Theorem~7.3]{lam2016total} which is proved for a \textit{simple} algebraic group of \textit{adjoint type}.
Our semisimple group $G$ splits into the product of simple components.
Correspondingly, $U^{\nil}_{\ge0}$ splits into the product of the ones lying in simple algebraic groups.
To apply Lam--Rietsch's result, we first investigate this splitting in Section~\ref{sec splittings}, and then we study the quotient by the center in Section~\ref{sec LR in our setting} to connect our group $G$ to that of adjoint type. These arguments will finally lead us to a proof of Theorem~\ref{thm LR in our setting}.

\vspace{30pt}

\section{Splittings of $U^{\nil}_{\ge0}$ and $\Delta_{\varpi_i}$}\label{sec splittings}

In this section, we show that we may assume that $G$ is a (simply connected) \textit{simple} algebraic group when we prove Theorem~\ref{thm LR in our setting}.
For this purpose, we describe the splittings of $U^e_{\ge0}$ and the functions $\Delta_{\varpi_i}$ $(i\in I)$.

\subsection{Splitting of the groups}

Recall that $I$ is the Dynkin diagram of the root system for $G$ with respect to the maximal torus $T$ and the Borel subgroup $B$.
Let 
\begin{align}\label{eq GJ splitting 10}
 I=I_1\sqcup\cdots\sqcup I_m
\end{align}
be the decomposition into connected components of $I$.
For each $1\le k\le m$, we have the simple algebraic group $G_{I_k}=(L_{I_k},L_{I_k})$ associated to the subset $I_k\subseteq I$ (Section~\ref{subsec Richardson strata of Y}). 
These subgroups are precisely the simple components of $G$ (e.g.\ \cite[Sect.~27.5]{humphreys2012linear}). Namely, we have $(G_{I_k},G_{I_{\ell}})=e$ for all $k\ne \ell$, and the group $G$ splits into the product of these subgroups:
\begin{align}\label{eq decomp into simple comp 1}
 G=G_{I_1}\cdots G_{I_m}.
\end{align}
In particular, the product map 
\begin{align*}
 G_{I_1}\times G_{I_2} \times \cdots \times G_{I_m} \rightarrow G
\end{align*}
has a finite kernel (loc.\ cit.). 
Recall that $B_{I_k} = B \cap G_{I_k}$ is a Borel subgroup of $G_{I_k}$ for $1\le k\le m$.

Recall that $U$ is  the unipotent radical of $B$. It also splits into the direct product
\begin{align}\label{eq GJ splitting 55}
 U = U_{I_1} \cdots U_{I_m}.
\end{align}
Here, each $U_{I_k}$ ($1\le k\le m$) is the product of root subgroups $U_{\alpha}$ for $\alpha\in \Phi_{I_k}^+$, being the unipotent radical of $B_{I_k}$.
Recall that 
\begin{align*}
 \nil = \sum_{i\in I} \nil_i = \nil_{I_1} + \cdots + \nil_{I_m},
\end{align*}
where $\nil_{I_k}=\sum_{i\in I_k} \nil_i$ is a regular nilpotent element in $\mathfrak{g}_{I_k}$ for $1\le k\le m$.
Note that each $G_{I_\ell}$ acts trivially on $\mathfrak{g}_{I_k}$ when $\ell\ne k$.
Hence, an element $u=u_1\cdots u_m\in U$ (with $u_k\in U_{I_k}$ for $1\le k\le m$) centralizes $e= \nil_{I_1} + \cdots + \nil_{I_m}\in \mathfrak{g}_{I_1}\oplus\cdots\oplus \mathfrak{g}_{I_m}$ if and only if each $u_k$ centralizes $\nil_{I_k}$ for $1\le k\le m$.
Therefore, the centralizer $U^e$ splits into the direct product, i.e., $U^{\nil} = U_{I_1}^{\nil_{I_1}} \cdots U_{I_m}^{\nil_{I_m}}$.
In particular, we obtain the product isomorphism
\begin{align}\label{eq GJ splitting 57}
 U_{I_1}^{\nil_{I_1}} \times \cdots \times U_{I_m}^{\nil_{I_m}} \rightarrow U^{\nil}.
\end{align}
Recalling \eqref{eq GJ splitting 10}, 
we also have the product isomorphism for the Weyl group: 
\begin{align}\label{eq GJ splitting 60}
 W_{I_1}\times \cdots \times W_{I_m} \rightarrow W.
\end{align}
Correspondingly, we have $w_0 = w_1 \cdots w_m$, where $w_k$ is the longest element of $W_{I_k}$ for $1\le k\le m$. 

\vspace{10pt}

\subsection{Splitting of the functions $\Delta_{\varpi_i}$}\label{subsec splitting of GJ and Delta}

Let us consider a map given by
\begin{align*}
 \Delta \colon G \rightarrow \C^{I}
 \quad ; \quad
 g \mapsto (\Delta_{\varpi_1}(g),\ldots,\Delta_{\varpi_{\rkg}}(g)),
\end{align*}
where $\Delta_{\varpi_i}$ is the function on $G$ defined in \eqref{eq def of Delta}. 
In this section, we show that the map $\Delta$ splits as well according to the splitting \eqref{eq decomp into simple comp 1} of $G$.

For $i\in I$, recall that $V_{\varpi_i}$ is an irreducible $G$-module, and let $v_{\varpi_i}\in V_{\varpi_i}$ be a highest weight vector of weight $\varpi_i$.
We fix $1\le k\le m$, and consider the simple component $G_{I_k}$ of $G$.
Let $i\in I_k$, and suppose that $G_{I_{\ell}}$ $(I_{\ell}\ne I_k)$ is another simple component of $G$. 
Since we have $i\notin I_{\ell}$, we have
\begin{align}\label{eq GIell acts trivially}
G_{I_{\ell}}v_{\varpi_i}=v_{\varpi_i} \quad \text{in $V_{\varpi_i}$}
\end{align}
from the proof of Proposition~\ref{prop restriction of Delta}~(ii).
For each $J\subseteq I$, we consider a map 
\begin{align*}
 \Delta_J \colon G_J\rightarrow \C^J
 \quad ; \quad
 g \mapsto (\Delta_{\varpi'_i}(g))_{i\in J},
\end{align*}
where $\Delta_{\varpi'_i}$ is the function on $G_J$ defined in \eqref{eq def of Delta'}.

\begin{proposition}\label{prop Deltas commute}
We have the following commutative diagrams.\vspace{5pt}
\[
  \xymatrix{
    G_{I_1}\times\cdots\times G_{I_m} \ar[rd]^{\qquad \Delta_{I_1}\times\cdots\times\Delta_{I_m}} \ar[d]_{\text{\rm prod.}} & \\
    G \ar[r]_{\Delta} & \C^{I}
  }
\]
\end{proposition}

\begin{proof}
Let us take $(y_1,\ldots,y_m)\in G_{I_1}\times\cdots\times G_{I_m}$.
Since $I=I_1\sqcup\cdots\sqcup I_m$, we have
\begin{align*}
 \Delta(y_1\cdots y_m) 
 &= ((y_1\cdots y_m v_{\varpi_i})_{\varpi_i})_{i\in I} \\
 &= \big(\ ((y_1\cdots y_m v_{\varpi_i})_{\varpi_i})_{i\in I_1}\ , \ \ldots \ , \ ((y_1\cdots y_m v_{\varpi_i})_{\varpi_i})_{i\in I_m} \ \big).
\end{align*}
Combining \eqref{eq GIell acts trivially} with the fact that $y_1,\ldots,y_m$ are pairwise commuting, we obtain that
\begin{align*}
 \Delta(y_1\cdots y_m) 
 = \big(\ ((y_1v_{\varpi_i})_{\varpi_i})_{i\in I_1}\ , \ \ldots \ , \ ((y_m v_{\varpi_i})_{\varpi_i})_{i\in I_m} \ \big) .
\end{align*}
Since we have $y_k\in G_{I_k}$ for $1\le k\le m$, this coincides with $\big( \Delta_{I_1}(y_1),\ldots,\Delta_{I_m}(y_m) \big)$ by Proposition~\ref{prop restriction of Delta}~(i).
\end{proof}

\vspace{10pt}

\subsection{Splitting of the totally nonnegative part of $U^e$}

Recall from \eqref{eq GJ splitting 55} that we have the product isomorphism $U_{I_1}\times \cdots \times U_{I_m} \stackrel{\cong}{\rightarrow} U$.
By the definitions of totally nonnegative parts of the groups appearing here (Section~\ref{TNN G and G/B}), this map further restricts to an isomorphism of monoids 
\begin{align*}
 (U_{I_1})_{\ge0}\times  \cdots \times (U_{I_m})_{\ge0} \stackrel{\cong}{\rightarrow} U_{\ge0}.
\end{align*}
Combining this with \eqref{eq GJ splitting 57}, we obtain the product isomorphism
\begin{align}\label{eq splitting of TN 20}
 (U_{I_1})^{\nil_{I_1}}_{\ge0}\times \cdots \times (U_{I_m})^{\nil_{I_m}}_{\ge0} \stackrel{\cong}{\rightarrow} U^e_{\ge0}.
\end{align}
Now we are ready to prove the following splitting.

\begin{proposition}\label{prop splitting of LR map}
The map 
\begin{align*}
 U^e_{\ge0} \rightarrow \R^{I}_{\ge0}
 \quad ; \quad
 x \mapsto 
 (\Delta_{\varpi_1}(x\dot{w}_0), \ldots, \Delta_{\varpi_{\rkg}}(x\dot{w}_0))
\end{align*}
appeared in Theorem~$\ref{thm LR in our setting}$ is identified with the product of the maps
\begin{align*}
 (U_{I_k})^{e_{I_k}}_{\ge0} \rightarrow \R^{I_k}_{\ge0}
 \quad ; \quad
 x \mapsto 
 (\Delta_{\varpi'_i}(x\dot{w}_{I_k}))_{i\in I_k}
 \qquad (1\le k\le m),
\end{align*}
where each $\Delta_{\varpi'_i}$ $(i\in I_k)$ is the function on $G_{I_k}$ defined in \eqref{eq def of Delta'}.
\end{proposition}

\begin{proof}
For $x\in U^e_{\ge0}$, we can write 
\begin{align*}
 x=x_1\cdots x_m
\end{align*}
for some $x_k\in (U_{I_k})^{e_{I_k}}_{\ge0}$ for $1\le k\le m$ by \eqref{eq splitting of TN 20}.
We also have $w_0=w_1\cdots w_m$ as we saw below \eqref{eq GJ splitting 60}, where each $w_k$ is the longest element of $W_{I_k}$.
This means that $\dot{w}_0=\dot{w}_1\cdots \dot{w}_m$ (\cite[Exercise~9.3.4]{springer1998linear}).
Hence, we obtain
\begin{align*}
 x\dot{w}_0 = (x_1\dot{w}_1)\cdots (x_m\dot{w}_m),
\end{align*}
where we have $x_k\dot{w}_k\in G_{I_k}$ for $1\le k\le m$ because of the choice of representative $\dot{w}_k$ given in \eqref{eq choice of representative 1} and \eqref{eq choice of representative 2}.
Now the claim follows from Proposition~\ref{prop Deltas commute}
\end{proof}

\vspace{10pt}

Since each $G_{I_k}$ is a (simply connected) simple algebraic group, Proposition~\ref{prop splitting of LR map} means that we may assume that our group $G$ is a (simply connected) \textit{simple} algebraic group when we prove Theorem~\ref{thm LR in our setting}.

\vspace{30pt}

\section{Proof of Theorem~\ref{thm LR in our setting}}\label{sec LR in our setting}

In this section, we give a proof of Theorem~\ref{thm LR in our setting} for an arbitrary simply connected semisimple algebraic group $G$ over $\C$.
For this purpose, we may assume that $G$ is \textit{simple} as we pointed out at the end of the last section.
To apply Lam--Rietsch's result \cite[Theorem~7.3]{lam2016total} (proved for a simple algebraic group of \textit{adjoint} type), we need to pass to the quotient $G/Z_G$. In the quotient, however, the fundamental weights $\varpi_i$ are no longer genuine characters of the maximal torus $T/Z_G$, and the definition of the functions $\Delta_{\varpi_i}$ needs some modification as we will see below. 

\subsection{Passing to the quotient by $Z_G$}

For each $i\in I$, there exists $m_i\in\Z_{>0}$ such that $m_i\varpi_i$ lies in the root lattice:
\begin{align*}
 m_i\varpi_i \in \bigoplus_{k\in I} \Z\alpha_{k} .
\end{align*}
We consider the irreducible representation $V_{m_i\varpi_i}$ of $G$ with highest weight $m_i\varpi_i$. Let $v_{m_i\varpi_i}\in V_{m_i\varpi_i}$ be a highest weight vector.

\begin{lemma}\label{lem Z acts trivially on Vmipii}
The center $Z_G\subseteq G$ acts on $V_{m_i\varpi_i}$ trivially.
\end{lemma}

\begin{proof}
Recall that 
\begin{align}\label{eq center = kernel alpha}
 Z_G = \bigcap_{\alpha\in\Phi} \Ker\alpha \subseteq T
\end{align}
(\cite[Sect.\ 26, Exercise~2]{humphreys2012linear}).
So $Z_G$ acts on $V_{m_i\varpi_i}$ via some weights of $T$.
An arbitrary weight $\lambda$ of $V_{m_i\varpi_i}$ can be written as 
\begin{align*}
 \lambda = m_i\varpi_i - \sum_{\alpha\in\Phi^+} c_{\alpha}\alpha 
\end{align*}
for some coefficients $c_{\alpha}\in\Z_{\ge0}$.
Since $m_i\varpi_i$ lies in the root lattice, it follows that $\lambda$ is an integral linear combination of roots. 
This and \eqref{eq center = kernel alpha} show that $\lambda (z) = 1$ for $z\in Z_G$.
This means that $Z_G$ acts on $V_{m_i\varpi_i}$ trivially.
\end{proof}

We set $\Gad \coloneqq G/Z_G$.
For $g\in G$, we write its right coset by $[g]\in \Gad$.
Since we are assuming that $G$ is a simple algebraic group, so is $\Gad$.
Lemma~\ref{lem Z acts trivially on Vmipii} implies the next claim.

\begin{corollary}\label{cor Z acts trivially on Vmipii}
$V_{m_i\varpi_i}$ has a structure of a $\Gad$-module given by
$$[g] v \coloneqq g v$$
for $[g]\in\Gad$ and $v\in V_{m_i\varpi_i}$.
\end{corollary}

\vspace{10pt}

Let 
\begin{align*}
 \pi \colon G\rightarrow \Gad
 \quad ; \quad
 g\mapsto [g]
\end{align*}
be the quotient map by $Z_G$.
Then $\Bad\coloneqq \pi(B)$ is a Borel subgroup of $\Gad$, and $\Tad\coloneqq \pi(T)$ is a maximal torus of $\Gad$, and $\Uad\coloneqq \pi(U)$ is a unipotent radical of $\Bad$ (\cite[Corollary~21.3~C]{humphreys2012linear}).
We also set $\Umad\coloneqq \pi(U^-)$.

$G$ and $\Gad$ have the same Lie algebra (identified under the differential of $\pi$ at the identity), and they have the same root system since the adjoint action of $T$ on $\mathfrak{g}$ factors through the quotient torus $\Tad$ (see \eqref{eq center = kernel alpha}). More precisely, for each root $\alpha\colon T\rightarrow \C^{\times}$ of $\Phi$, there is a unique homomorphism $\alpha'\colon \Tad\rightarrow \C^{\times}$ satisfying $\alpha'\circ\pi=\alpha$.
To simplify the notation, we write $\alpha\colon \Tad\rightarrow \C^{\times}$ instead of $\alpha'$ for the rest of this section.
Then, for $i\in I$, we have
\begin{align*}
\alpha_i\colon \Tad\rightarrow \C^{\times}
\quad \text{and} \quad
 \alpha^{\vee}_i\colon \C^{\times}\rightarrow \Tad ,
\end{align*}
where the second map is the composition of $\alpha^{\vee}_i\colon \C^{\times}\rightarrow T$ and the quotient map $T\rightarrow \Tad$.
Note that the differentials of these maps coincide with the maps in \eqref{eq differential of roots and coroots}.
By \eqref{eq center = kernel alpha}, the intersection of the kernels of roots for $\Gad$ is trivial which means that $\Gad$ is of adjoint type.

The restrictions
\begin{align}\label{eq U iso Uad}
 \pi|_{U} \colon U \rightarrow \Uad
 \quad \text{and} \quad
 \pi|_{U^-} \colon U^-\rightarrow \Umad
\end{align}
are isomorphisms of algebraic groups since we have $U\cap Z_G=\{e\}=U^-\cap Z_G$.
Since both of these maps commute with the conjugations of $T$ and $\Tad$ under the map $\pi|_T\colon T\rightarrow \Tad$, it follows that these maps preserve root subgroups (e.g.\ \cite[Sect.~26.3]{humphreys2012linear}). 

\subsection{Pinning for $\Gad$}

For $i\in I$, we set
\begin{align}\label{eq Chevalley'}
 e'_i\coloneqq -e_i\in\mathfrak{g}_{\alpha_i}
 \quad \text{and} \quad
 f'_i\coloneqq -f_i\in\mathfrak{g}_{-\alpha_i}.
\end{align}
We take them as our choice of Chevalley generators for $\Gad$, and we define parametrizations $\xad_i \colon \C\rightarrow \Uad_{\alpha_i}$ and $\yad_i \colon \C\rightarrow \Uad_{-\alpha_i}$ for $i\in I$ by
\begin{align}\label{eq pinning' for Gad}
  \xad_i(t) = \pi(\exp(te'_i)) \quad \text{and} \quad \yad_i(t)=\pi(\exp(tf'_i))
  \qquad (t\in\C^{\times}).
\end{align}

\vspace{5pt}

\begin{definition}\label{def pinning of Gad}
We define $\Gad_{\ge0}$ $($resp.\ $\Umad_{\ge0}$$)$ to be the totally nonnegative part of $\Gad$ and $($resp.\ $\Umad$$)$ with respect to the pinning given in \eqref{eq pinning' for Gad}. 
More precisely, \vspace{3pt}
\begin{itemize}
 \item[$(1)$] $\Gad_{\ge0}$ is the subomonoid of $\Gad$ generated by $\xad_i(t)$, $\yad_i(t)$ for $i\in I$ and $t\in\R_{\ge0}$ and by $\chi(t)$ for $\chi\in\Hom(\C^{\times},\Tad)$ and $t\in\R_{>0}$, \vspace{5pt}
 \item[$(2)$] $\Umad_{\ge0}$ is the subomonoid of $\Gad$ generated by $\yad_i(t)$ for $i\in I$ and $t\in\R_{\ge0}$.\vspace{3pt}
\end{itemize}
$($We will explain the reason for using the \textit{negative} Chevalley generators \eqref{eq Chevalley'} below.$)$
\end{definition}

\vspace{5pt}

For each $w\in W$, we define $\mathring{w}\in N_{\Gad}(\Tad)$ by the same manner as in \eqref{eq choice of representative 1} and \eqref{eq choice of representative 2} but with the pinning \eqref{eq pinning' for Gad}.
Then, by definition, we have $\mathring{s}_i=[\dot{s}_i^{-1}]$ in $\Gad$ for $i\in I$. 
For the longest element $w_0\in W$, we have $w_0=w_0^{-1}$ from which one can deduce that
\begin{align*}
 \mathring{w}_0=[\dot{w}_0^{-1}] \quad \text{in $N_{\Gad}(\Tad)$},
\end{align*}
where $\dot{w}_0^{-1}\coloneqq (\dot{w}_0)^{-1}$.

\vspace{10pt}

\subsection{Lam--Rietsch's function $\Delta_i$}

Recall that $\Gad$ acts on $V_{m_i\varpi_i}$ (Corollary~\ref{cor Z acts trivially on Vmipii}).
Let $v_{m_i\varpi_i}$ be a highest weight vector in $V_{m_i\varpi_i}$, and $\shapo{\ }{\ }$ the Shapovalov form on $V_{m_i\varpi_i}$ normalized by the condition $\shapo{v_{m_i\varpi_i}}{v_{m_i\varpi_i}}=1$.
This form is characterized by the contravariance with respect to the Chevalley generators $e_i$, $f_i$ (hence $e'_i$, $f'_i$ as well) for $i\in I$ (\cite[Sect.\ 3.14]{humphreys2012linear}).
Let us take a lowest weight vector in $V_{m_i\varpi_i}$:
\begin{align}\label{eq choice of lowest weight vector}
 v^-_{m_i\varpi_i} \coloneqq \mathring{w}_0v_{m_i\varpi_i} = \dot{w}_0^{-1}v_{m_i\varpi_i} \in V_{m_i\varpi_i}.
\end{align}
We set
\begin{align*}
 f \coloneqq\sum_{i\in I} f'_i = -\sum_{i\in I} f_i,
 \quad \text{and} \quad 
 (\Gad)^f \coloneqq \Gad\cap Z_{\Gad}(f).
\end{align*}
Note that we have $(\Gad)^f = (\Umad)^f$ since $\Gad$ is of adjoint type (e.g.\ \cite[Sect.~2]{Bualibanu2017Peterson}). So we set
\begin{align*}
 (\Gad)^f_{\ge0} \coloneqq (\Gad)^f \cap (\Umad)_{\ge0}.
\end{align*}
For each $i\in I$, let us consider a function $\Delta_i \colon (\Gad)^f_{\ge0} \rightarrow \C$ given by
\begin{align*}
 \Delta_i([y]) \coloneqq \shapo{[y]v_{m_i\varpi_i}}{v^-_{m_i\varpi_i}} = \shapo{yv_{m_i\varpi_i}}{v^-_{m_i\varpi_i}}
 \qquad ([y]\in (\Gad)^f_{\ge0} ).
\end{align*}
Note that the value $\Delta_i([y])$ does not depend on the choice of the highest weight vector $v_{m_i\varpi_i}$. 
This is essentially the function appeared in the work of Lam--Rietsch (\cite[Section~12]{lam2016total}) (see Section~12 and Remark~6.5 of loc.\ cit.). In fact, they chose a specific lowest weight vector of $V_{m_i\varpi_i}$. We will see below that our choice of $v^-_{m_i\varpi_i}$ agrees with that of Lam--Rietsch up to a \textit{positive} scalar multiple.

\vspace{10pt}

\subsection{Relation to our function $\Delta_{\varpi_i}$}

In this section, we clarify the relation between the two functions $\Delta_i \colon (\Gad)^f_{\ge0} \rightarrow \C$ and $\Delta_{\varpi_i} \colon G^{\nil}_{\ge0} \rightarrow \C$, where the latter is the restriction of the function defined in \eqref{eq def of Delta}.
There are two main differences; (1) the groups $\Gad$ and $G$, (2) nilpotent elements $f$ and $e$.
The difference (1) can be treated by the quotient map $\pi\colon G\rightarrow \Gad$.
To connect the difference (2), we use the longest element $w_0$ which with Lemma~\ref{lemm w0e=-f first} requires the choice of our Chevalley generators $e'_i, f'_i$ as we will see below.
We begin with the following well-known property of the Shapovalov form.

\begin{lemma}\label{lem property of Shapovalov form}
For $w\in W$, $v,v'\in V_{m_i\varpi_i}$, we have
\begin{align*}
 \shapo{\dot{w}v}{\dot{w}v'} = \shapo{v}{v'} .
\end{align*}
$($The same equality works for $\mathring{w}_0$ as well.$)$
\end{lemma}

\begin{proof}
We first prove the claim for the case $w=s_i$ for some $i\in I$. Recall from \eqref{eq choice of representative 1} that we have
\begin{align*}
 \dot{s}_i = y_i(1)x_i(-1)y_i(1) = x_i(-1)y_i(1)x_i(-1).
\end{align*}
So we have
\begin{align*}
 \shapo{\dot{s}_iv}{\dot{s}_iv'} 
 = \shapo{y_i(1)x_i(-1)y_i(1)v}{x_i(-1)y_i(1)x_i(-1)v'} 
 = \shapo{v}{v'} ,
\end{align*}
where the second equality follows from the contravariance of the Shapovalov form.
The assertion for a general $w\in W$ now follows by repeating this equality.
\end{proof}

\vspace{10pt}

Recall from \eqref{eq U iso Uad} that we have an isomorphism $U \rightarrow \Uad$ given by $x \mapsto [x]$.
Thus, we obtain an isomorphism
\begin{align*}
 U \rightarrow \Umad
 \quad ; \quad 
 x \mapsto [\dot{w}_0^{-1}x\dot{w}_0].
\end{align*}
By Definition~\ref{def pinning of Gad} and $\Ad_{\dot{w}_0^{-1}} e_i = - f_{i^*}$ for $i\in I$ (Lemma~\ref{lemm w0e=-f first}), this map restricts to an isomorphism
\begin{align*}
 U_{\ge0} \rightarrow \Umad_{\ge0}
 \quad ; \quad 
 x \mapsto [\dot{w}_0^{-1}x\dot{w}_0].
\end{align*}
Thus, we obtain an isomorphism\vspace{5pt}
\begin{align}\label{eq Uene and Uadene}
 U^e_{\ge0} \rightarrow (\Gad)^f_{\ge0}
 \quad ; \quad 
 x \mapsto [\dot{w}_0^{-1}x\dot{w}_0]
\end{align}
\vspace{-5pt}\\
since $U^e_{\ge0}=U^e\cap U_{\ge0}$ and $(\Gad)^f_{\ge0}= (\Gad)^f\cap \Umad_{\ge0}$.

\vspace{20pt}

Based on the isomorphism \eqref{eq Uene and Uadene}, the next claim states that $\Delta_i$ is essentially a power of $\Delta_{\varpi_i}$.

\begin{proposition}\label{prop LR and our}
Let $i\in I$. For all $x\in U^e_{\ge0}$, we have
\begin{align*}
 \Delta_i([\dot{w}_0^{-1}x\dot{w}_0]) = (\Delta_{\varpi_i}(x\dot{w}_0))^{m_i}.
\end{align*}
\end{proposition}

\begin{proof}
We have
\begin{equation}\label{eq LR 10 rewrite}
\begin{split}
 \Delta_i([\dot{w}_0^{-1}x\dot{w}_0]) 
 &= \shapo{\dot{w}_0^{-1}x\dot{w}_0v_{m_i\varpi_i}}{v^-_{m_i\varpi_i}} \\ 
 &= \shapo{\dot{w}_0^{-1}x\dot{w}_0v_{m_i\varpi_i}}{\dot{w}_0^{-1}v_{m_i\varpi_i}} \quad \text{(by \eqref{eq choice of lowest weight vector})} \\
 &= \shapo{x\dot{w}_0v_{m_i\varpi_i}}{v_{m_i\varpi_i}} \quad \text{(by Lemma~\ref{lem property of Shapovalov form})}.
 \end{split}
\end{equation} 
Recalling our normalization $\shapo{v_{m_i\varpi_i}}{v_{m_i\varpi_i}}=1$, we can write the value of the Shapovalov form appearing in the last expression as follows:
\begin{equation}\label{eq LR rewrite 20}
 \begin{split}
 \shapo{x\dot{w}_0v_{m_i\varpi_i}}{v_{m_i\varpi_i}} 
 =\text{the coeff.\ of $v_{m_i\varpi_i}$ for the w.d.\ of $x\dot{w}_0v_{m_i\varpi_i}$ in $V_{m_i\varpi_i}$},
 \end{split}
\end{equation}
where w.d.\ is the abbreviation of ``weight decomposition".
We investigate this coefficient below by realizing $V_{m_i\varpi_i}$ in the tensor product $(V_{\varpi_i})^{\otimes m_i}$, where $V_{\varpi_i}$ is the irreducible $G$-module of highest weight $\varpi_i$.

We consider the $G$-module $(V_{\varpi_i})^{\otimes m_i}$ with a highest weight vector $(v_{\varpi_i})^{\otimes m_i}$. This might not be irreducible, but $(v_{\varpi_i})^{\otimes m_i}$ is a highest weight vector of weight $m_i\varpi_i$. Let 
\begin{align*}
 V\coloneqq \text{span}_{\C} \{ G\cdot (v_{\varpi_i})^{\otimes m_i} \}\subseteq (V_{\varpi_i})^{\otimes m_i}.
\end{align*} 
Then $V$ is the irreducible $G$-submodule whose highest weight is $m_i\varpi_i$ (see the proof of Proposition~\ref{prop restriction of Delta}~(i)).
This implies that we have 
\begin{align*}
 (v_{\varpi_i})^{\otimes m_i}\in V_{m_i\varpi_i}\subseteq (V_{\varpi_i})^{\otimes m_i}.
\end{align*} 
Notice that the coefficient in \eqref{eq LR rewrite 20} does not depend on the choice of the highest weight vector $v_{m_i\varpi_i}$.
So, without loss of generality, we may assume that $v_{m_i\varpi_i} = (v_{\varpi_i})^{\otimes m_i}$
when we compute the coefficient in \eqref{eq LR rewrite 20}.
Therefore, 
\begin{align*}
\text{RHS of }\eqref{eq LR rewrite 20}
 &= \text{the coeff.\ of $(v_{\varpi_i})^{\otimes m_i}$ for the w.d.\ of $x\dot{w}_0(v_{\varpi_i})^{\otimes m_i}$ in $(V_{\varpi_i})^{\otimes m_i}$} \\
 &= \Big(\text{the coeff.\ of $v_{\varpi_i}$ for the w.d.\ of $x\dot{w}_0v_{\varpi_i}$ in $V_{\varpi_i}$} \Big)^{m_i} \\
 &= (\Delta_{\varpi_i}(x\dot{w}_0))^{m_i} \qquad \text{(by the definition \eqref{eq def of Delta})},
\end{align*} 
where for the second equality we used the fact that $v_{\varpi_i}$ 
is a highest 
weight vector in $V_{\varpi_i}$. 
Combining this with \eqref{eq LR 10 rewrite} and \eqref{eq LR rewrite 20}, we obtain
\begin{align*}
 \Delta_i([\dot{w}_0^{-1}x\dot{w}_0]) =(\Delta_{\varpi_i}(x\dot{w}_0))^{m_i} ,
\end{align*}
as desired.
\end{proof}

\vspace{10pt}

Recalling that we have $U^e_{\ge0} \stackrel{\cong}{\rightarrow} (\Gad)^f_{\ge0}$ given by $x \mapsto [\dot{w}_0^{-1}x\dot{w}_0]$ from \eqref{eq Uene and Uadene}, this proposition implies that 
\begin{align*}
 \Delta_i([y]) = \shapo{yv_{m_i\varpi_i}}{v^-_{m_i\varpi_i}} \ge 0
\end{align*}
for all $y\in (\Gad)^f_{\ge0}$ by Proposition~\ref{prop Delta is nonnegative}.
As we discussed at the end of the last section, our choice of the lowest weight vector $v^-_{m_i\varpi_i}$ might not agree with that of Lam--Rietsch. But their choice satisfies the same nonnegativity (\cite[Theorem~7.3]{lam2016total}). Therefore, the two choices must agree up to a positive scalar multiple.
Now we can state the following result of Lam--Rietsch in terms of $\Delta_i$. 

\begin{theorem}\label{thm Lam-Rietsch}
$($\cite[Theorem~7.3]{lam2016total}$)$
The map 
\begin{align*}
 (\Gad)^f_{\ge0} \rightarrow \R^{I}_{\ge0}
 \quad ; \quad 
 [y]\mapsto (\Delta_i([y]))_{i\in I}
\end{align*}
is a homeomorphism.
\end{theorem}

\vspace{10pt}

\subsection{Proof of Theorem~\ref{thm LR in our setting}}

We now give a proof of Theorem~\ref{thm LR in our setting} (for a simply connected semisimple algebraic group $G$). For the reader's convenience, we state the claim here again.\vspace{10pt}

\noindent
\textbf{Theorem.}\ 
\textit{
The map
\begin{align*}
 U^e_{\ge0} \rightarrow \R^{I}_{\ge0}
 \quad ; \quad
 x \mapsto (\Delta_{\varpi_1}(x\dot{w}_0),\ldots,\Delta_{\varpi_{\rkg}}(x\dot{w}_0))
\end{align*}
is a homeomorphism.
}

\vspace{0pt}

\begin{proof}
As we saw at the end of Section~\ref{sec splittings}, we may assume that $G$ is a (simply connected) simple algebraic group so that $\Gad$ is simple algebraic group of adjoint type.

We first consider a map with \textit{powers} $m_i$:
\begin{align}\label{eq the map with powers}
 U^e_{\ge0} \rightarrow \R^{I}_{\ge0}
 \quad ; \quad
 x \mapsto ((\Delta_{\varpi_i}(x\dot{w}_0))^{m_i})_{i\in I}.
\end{align}
By \eqref{eq Uene and Uadene} and Proposition~\ref{prop LR and our}, this is identified with a map
\begin{align*}
 (\Gad)^f_{\ge0} \rightarrow \R^{I}_{\ge0}
 \quad ; \quad
 [y] \mapsto (\Delta_{i}([y]))_{i\in I}.
\end{align*}
It now follows from Lam--Rietsch's result (Theorem~\ref{thm Lam-Rietsch}) that this map (and hence \eqref{eq the map with powers} as well) is a homeomorphism.
We now consider the map in the claim 
(without powers).
It is the composition of the map \eqref{eq the map with powers} and the map
\begin{align*}
 \R^{I}_{\ge0} \rightarrow \R^{I}_{\ge0}
 \quad ; \quad
 (a_i)_{i\in I} \mapsto ((a_i)^{1/m_i})_{i\in I}.
\end{align*}
Since the latter map is a homeomorphism, the claim follows.
\end{proof}

\vspace{30pt}

\section{Appendix}

\subsection{Proof of Lemma~\ref{lem closed embedding of TN 3}}\label{subsec Appendix 2}

We give a proof of Lemma~\ref{lem closed embedding of TN 3}.
For the reader's convenience, we state the claim here again.

\vspace{10pt}
\noindent
\textbf{Lemma.}
\textit{Under the closed embedding $G_J/B_J \hookrightarrow G/B$, the image of $(G_J/B_J)_{\ge0}$ is precisely $G_J/B_J\cap (G/B)_{\ge0}$.}

\begin{proof}
Let $g_J\in G_J$. What we need to prove is the following equivalence:
\begin{equation}\label{eq wwntp}
 g_JB_J \in (G_J/B_J)_{\ge0} \quad \Longleftrightarrow \quad g_JB \in (G/B)_{\ge0}.
\end{equation}
If $g_JB_J \in (G_J/B_J)_{\ge0}$, then we have
\begin{equation*}
g_JB_J\in \overline{(G_J/B_J)_{>0}}=\overline{(G_J)_{>0}B_J/B_J}.
\end{equation*}
Hence, under the embedding $G_J/B_J \hookrightarrow G/B$, we obtain 
\begin{equation*}
g_JB \in \overline{(G_J)_{>0}B/B}.
\end{equation*}
Since $(G_J)_{>0} \subseteq (G_J)_{\ge0} \subseteq G_{\ge0}$, it follows that 
\begin{equation*}
g_JB \in \overline{G_{\ge0}B/B} \subseteq \overline{(G/B)_{\ge0}} = (G/B)_{\ge0}.
\end{equation*}

To prove the opposite implication in \eqref{eq wwntp}, suppose that $g_JB\in (G/B)_{\ge 0}$.
Since $g_J\in G_J$, it follows from the Bruhat decomposition of $G_J/B_J$ that we may write
\begin{equation}\label{eq gjbj = xwbj}
 g_J B_J = x\dot{w}B_J
 \qquad \text{in $G_J/B_J$}
\end{equation}
for some $x\in U_J$ and $w\in W_J$, where we note that $\dot{w}\in G_J$ because of our choice of representatives given by \eqref{eq choice of representative 1} and \eqref{eq choice of representative 2}. We then have
\begin{equation*}
 g_J B = x\dot{w}B
 \qquad \text{in $G/B$}.
\end{equation*}
Let $t\in\R_{>0}$, and let $\xi=\exp(te)\in U$. Then we have $\xi\in U_{>0}$ by \cite[Proposition~5.9]{lusztig1994total}.
Hence, by $x\dot{w}B(=g_JB)\in (G/B)_{\ge 0}$ and Lemma~\ref{lem two facts for TNN}~(1), it follows that
\begin{equation*}
\xi x\dot{w}B
\in (G/B)_{\ge0} \cap (B^-eB/B).
\end{equation*}
Since $\xi x\dot{w}B \in B\dot{w}B/B$, we further obtain
\begin{equation*}
\xi x\dot{w}B
\in (G/B)_{\ge0} \cap (B^-eB/B)\cap (B\dot{w}B/B)=U^-(w)B/B,
\end{equation*}
where the last equality follows from Lemma~\ref{lem two facts for TNN}~(2).
Hence, we have
\begin{equation}\label{eq xi x w = u-b}
\xi x\dot{w} = u^- b
\end{equation}
for some $u^- \in U^-(w)\subseteq (U^-_J)_{\ge0}$ (since $w\in W_J$) and $b\in B$.
Let $P_J$ be the standard parabolic subgroup of $G$ associated to $J\subseteq I$ satisfying $B\subseteq P_J$. It admits the semidirect decomposition $P_J=V_JL_J$, where $V_J\subseteq U$ is the unipotent radical of $P_J$. 
Let us write 
\begin{equation*}
\xi = \xi_1\cdot \xi_2
\qquad (\xi_1\in V_J\subseteq U, \ \xi_2\in L_J).
\end{equation*}
Then we have $\xi_2\in U\cap L_J=U_J$.
Similarly, we write 
\begin{equation*}
b = b_1\cdot b_2
\qquad (b_1\in V_J\subseteq B, \ b_2\in L_J)
\end{equation*}
from which we have $b_2\in B\cap L_J$.
Now we obtain from \eqref{eq xi x w = u-b} that
\begin{equation*}
\xi_1\cdot (\xi_2 x\dot{w}) = u^- b_1 b_2 = b'_1 \cdot(u^- b_2)
\end{equation*}
for some $b'_1\in V_J$ since $u^- \in U^-_J\subseteq P_J$ and $V_J$ is normal in $P_J$.
Therefore, from the semidirect decomposition $P_J=V_JL_J$, we obtain
\begin{equation*}
 \xi_2 x\dot{w} = u^- b_2
 \qquad \text{in $L_J$}.
\end{equation*}
Since $\xi_2\in U_J$, we see that $b_2=(u^-)^{-1}\xi_2 x\dot{w}\in G_J$ so that $b_2\in B\cap G_J=B_J$.
Hence,
\begin{equation*}
 \xi_2 x\dot{w} B_J = u^- B_J
 \qquad \text{in $G_J/B_J$}.
\end{equation*}
Here, the right hand side lies in $(G_J/B_J)_{\ge0}$ since $u^-\in (U^-_J)_{\ge0}$ as we saw above. Thus, we have
\begin{equation*}
 \xi_2x\dot{w} B_J \in (G_J/B_J)_{\ge0}.
\end{equation*}
Note that we have $\xi_2 = p(\xi)=p(\exp(te))$, where $p\colon P_J\rightarrow L_J$ is the canonical projection.
By taking $t\rightarrow0$, we have $\xi_2\rightarrow 1$ which implies that
\begin{equation*}
 x\dot{w} B_J \in (G_J/B_J)_{\ge0}
\end{equation*}
since $(G_J/B_J)_{\ge0}$ is closed in $G_J/B_J$.
Recalling \eqref{eq gjbj = xwbj}, we now obtain the desired conclusion $g_JB_J\in (G_J/B_J)_{\ge0}$.
\end{proof}

\vspace{10pt}

\subsection{Proof of the second claim of Lemma~\ref{eq alpha surjectivity}}

We give a proof of the second claim of Lemma~\ref{eq alpha surjectivity}.
For the reader's convenience, we state Lemma~\ref{eq alpha surjectivity} here again.
Recall that $(T_J)_{>0}$ is the subgroup of $T_J$ generated by $\alpha^{\vee}_i(z)$ for $i\in I$ and $z\in\R_{>0}$.

\vspace{10pt}
\noindent
\textbf{Lemma.}
\textit{For $t\in (T_J)_{>0}$, we have $\varpi_i(t),\alpha_i(t)\in\R_{>0}$ for all $i\in I$.
Moreover, we have isomorphisms
\begin{align*}
 &(T_J)_{>0} \rightarrow (\R_{>0})^J
 \quad ; \quad
 t\mapsto (\varpi_i(t))_{i\in J} , \\
 &(T_{J})_{>0}\rightarrow (\R_{>0})^{J}
 \quad ; \quad
 t\mapsto (\alpha_i(t))_{i\in J}.
\end{align*} }

\begin{proof}
Recall that the first claim follows from the description of $(T_J)_{>0}$ above.
For the second claim, it suffices to prove the claim for the case $J=I$ (so that $T_J=T$) since the claim for a general case can be deduced by applying such result to the maximal torus $T_J$ of $G_J$.
Based on the first claim, we consider the homomorphisms
\begin{align}\label{eq varpi iso 5}
 &T_{>0}\rightarrow (\R_{>0})^{I}
 \quad ; \quad
 t\mapsto (\varpi_1(t),\ldots,\varpi_{\rkg}(t)), \\ \label{eq varpi iso 3}
 &T_{>0}\rightarrow (\R_{>0})^{I}
 \quad ; \quad
 t\mapsto (\alpha_1(t),\ldots,\alpha_{\rkg}(t)).
\end{align} 
We claim that \eqref{eq varpi iso 5} is a bijection.
This can be explained as follows. By the definition of $T_{>0}$ (see above), we have a map
\begin{align*}
 (\R_{>0})^{I} \rightarrow T_{>0}
 \quad ; \quad
 (z_1,\ldots,z_{\rkg}) \mapsto \alpha^{\vee}_1(z_1)\cdots \alpha^{\vee}_m(z_1),
\end{align*}
and this is the inverse map of \eqref{eq varpi iso 5} since the fundamental weights and simple coroots are dual to each other. In particular, \eqref{eq varpi iso 5} is a bijection as claimed above.

Let $C=(c_{i,j})_{i,j\in I}$ be the Cartan matrix of $\Phi$ given by $c_{i,j}=\langle\alpha_i,\alpha^{\vee}_j\rangle$ for $i,j\in I$.
This gives us a map
\begin{align}\label{eq varpi iso 6}
 (\R_{>0})^{I} \rightarrow (\R_{>0})^{I}
 \quad ; \quad 
 (z_1,\ldots,z_{\rkg}) \mapsto (z^{c_1},\ldots, z^{c_{\rkg}}),
\end{align}
where each $z^{c_i}\coloneqq z_1^{c_{i,1}}\cdots z_{\rkg}^{c_{i,\rkg}}$ is the monomial whose exponents are given by the $i$-th row vector $c_i$ of $C$.
By construction, this map sends $(\varpi_1(t),\ldots,\varpi_{\rkg}(t))$ to $(\alpha_1(t),\ldots,\alpha_{\rkg}(t))$ for all $t\in T_{>0}$.
Hence, the map \eqref{eq varpi iso 3} is the composition of \eqref{eq varpi iso 5} and \eqref{eq varpi iso 6}. So it suffices to show that \eqref{eq varpi iso 6} is a bijection.
Under the identification $\R_{>0}\cong\R$ given by the logarithm, \eqref{eq varpi iso 6} is identified with the linear map
\begin{align*}
 \R^{I} \rightarrow \R^{I}
 \quad ; \quad
 x \mapsto Cx.
\end{align*}
Since we have $\det C\ne0$, this is a bijection. Hence, the claim follows.
\end{proof}

\vspace{30pt}

\bibliographystyle{amsplain}
\bibliography{main}

\end{document}